\documentclass[12pt,draftcls,onecolumn]{IEEEtran}
\ifCLASSINFOpdf
\else
\fi
\usepackage{makeidx}  % allows for indexgeneration
\usepackage{amsfonts}
\usepackage{amsmath}
\usepackage{graphicx}
\usepackage{latexsym}
\usepackage{amssymb}
\usepackage[lined,boxed]{algorithm2e}

\newtheorem{definition}{Definition}
\newtheorem{proposition}{Proposition}
\newtheorem{theorem}{Theorem}
\newtheorem{lemma}{Lemma}
\newtheorem{corollary}{Corollary}
\newtheorem{example}{Example}

\begin{document}
\title{Linear time logic control of linear systems with disturbances}

\author{Jinjin~Zhang,~
        Zhaohui~Zhu,~
        and~Jianfei~Yang% <-this % stops a space
\thanks{This work received financial
support of the National Natural Science of China (No. 60973045), the NSF of Jiangsu Province
(No. BK2007191) and Fok Ying-Tung Education Foundation.}
\thanks{Jinjin~Zhang is with the Department
of Computer Science, Nanjing University of Aeronautics and Astronautics, Nanjing,  210016 China (e-mail: jinjinzhang@nuaa.edu.cn).}% <-this % stops a space
\thanks{Corresponding author. Zhaohui~Zhu is with Department of Computer Science, Nanjing University of Aeronautics and Astronautics, Nanjing,  210016 China; and is with State Key Lab of Novel Software Technology, Nanjing University, 210093 China (e-mail: bnj4892856@jlonline.com).}
\thanks{Jianfei~Yang is with the Department of Automation Engineering, Nanjing University of Aeronautics and Astronautics, Nanjing,  210016 China (e-mail: yjfsmile@nuaa.edu.cn).}% <-this % stops a space
}


\markboth{Journal of IEEE Transactions on Automatic Control}%
{Shell \MakeLowercase{\textit{et al.}}: Bare Demo of IEEEtran.cls for Journals}



\maketitle
\begin{abstract}
%\boldmath
The formal analysis and design of control systems is one of recent trends in control theory.
In this area, in order to reduce the complexity and scale of control systems, finite abstractions of control systems are introduced and explored.
In non-disturbance case, the controller of control systems is often generated from the controller of finite abstractions.
Recently, Pola and Tabuada provide approximate finite abstractions for linear control systems with disturbance inputs.
However, these finite abstractions and original linear systems do not always share the identical specifications,
which obstructs designing controller (of linear systems) based on their finite abstractions.
This paper tries to bridge such gap between linear systems and their finite abstractions.
\end{abstract}


% \fi
\IEEEpeerreviewmaketitle

\section{Introduction}\label{Sec:introduction}
In recent years, there has been an increasing interest in the formal analysis and design of control systems.
The formal analysis aims to check whether a control system satisfies desired specifications, while the formal design wants to construct a controller for control system so that it meets a given specification.
Early work in these fields is chiefly concerned with stability and reachability~\cite{hab:1,hab:2}. Recently, more complex specifications are considered. These specifications may be described by such as temporal logic \cite{alur,anton:1,fain:1,fain:2,kloe:1,tab:1}, regular expressions \cite{Koutsoukos} and transition systems \cite{tab:2}.
Amongst, temporal logic, due to its resemblance to natural language and the existence of algorithms for model checking, is widely adopted to describe the desired properties.
For example, linear temporal logic (LTL) is used to express specifications of discrete-time linear systems~\cite{tab:1} and continuous-time linear systems \cite{kloe:1}. Both Computation Tree Logic (CTL)\cite{anton:1} and LTL\cite{fain:1,fain:2} are adopted to specify task of mobile robotics.

In the formal analysis and design, it is always difficult to deal with large-scale control systems because of the complexity and scale of such systems.
To overcome this defect, finite abstractions are extracted from  these control systems.
For instance, Tabuada and Pappas explore finite abstractions of discrete-time linear systems and present some critical properties of linear systems ensuring the existence of finite abstractions \cite{tab:3}. Based on finite partitions of the set of inputs or outputs, finite symbolic models are constructed for nonlinear control systems in \cite{tab:4}.
A number of work has been devoted to finite abstractions of hybrid systems~\cite{henz,henz:1,alur:2,alur:3,laffer}. An excellent review of these work may be found in~\cite{alur}.

Finite abstractions play an important role in the formal design of control systems \cite{fain:2,kloe:1,tab:1,tab:2}.
As an example, Fig~\ref{fig:illust0} illustrates the function of finite abstraction in the formal design of linear system \cite{tab:1}.
Given a linear system $\Sigma$, Tabuada and Pappas provide an infinite transition system $T_{\Sigma}$ as the formal model of $\Sigma$ and construct a finite transition system $T_{\Sigma}^{f}$ as the finite abstraction of $\Sigma$.
The following result is a fundamental result in~\cite{tab:1}, which lays the foundation of the design method of controllers presented in \cite{tab:1}.
\begin{center}
\textit{$T_{\Sigma}$ and $T_{\Sigma}^{f}$ are bisimilar and share the same properties describe by linear temporal logic.}~($*$)
\end{center}
Thus, given an LTL specification $\varphi_0$, the formal design of $T_{\Sigma}$ can be equivalently performed on the finite abstraction $T_{\Sigma}^{f}$.
Tabuada and Pappas construct a controller $T_c$ of $T_{\Sigma}^{f}$ enforcing $\varphi_0$ and demonstrate that $T_{\Sigma}$ satisfies $\varphi_0$ under this controller as well.
Furthermore, based on this controller, a close-loop system $H$ satisfying $\varphi_0$ is generated.
Similar methods are also adopted in \cite{fain:2,kloe:1,tab:2}.
\begin{figure}[t]
\begin{center}
\centerline{\includegraphics[scale=0.6]{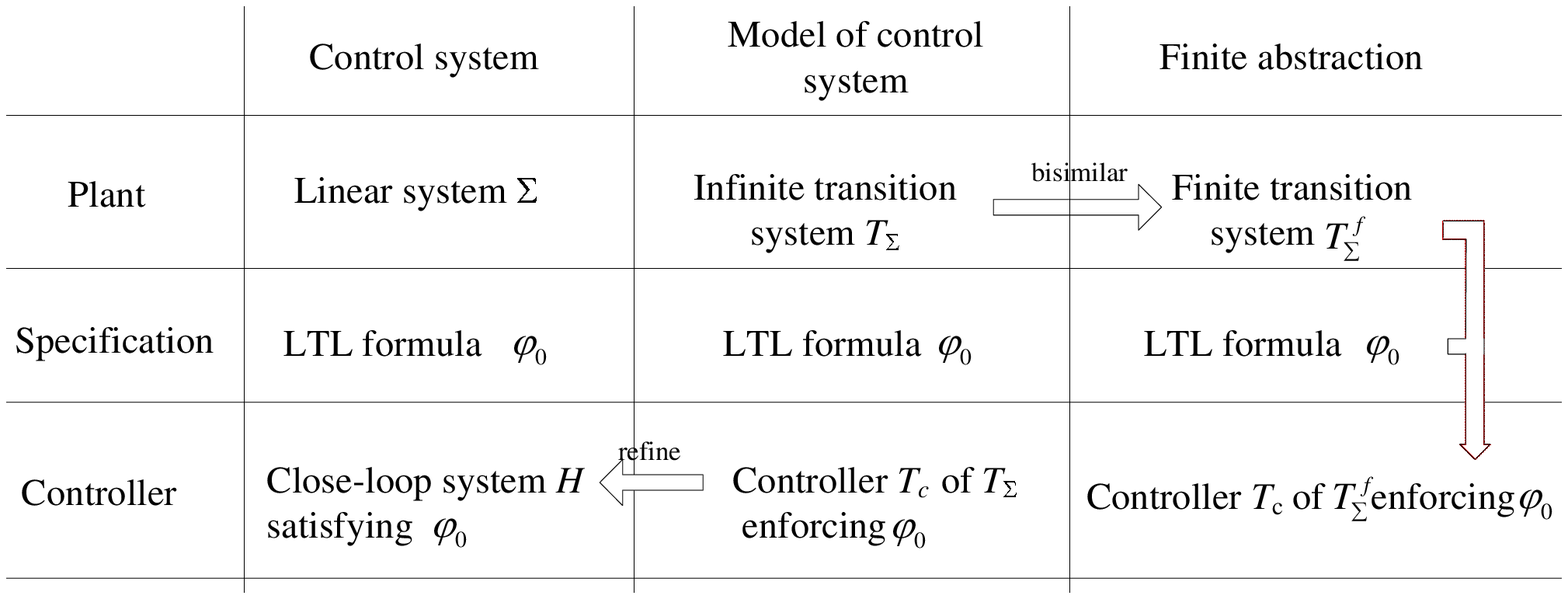}}\vspace{-0.6in}
\end{center}
\caption{Controller design \cite{tab:1}: non-disturbance case}\label{fig:illust0}
\end{figure}

The research work, mentioned above, focuses on control systems without reference to disturbances.
However, all physical systems are subject to some types of extraneous disturbances or noise during operation \cite{Kuo}.
In~\cite{pola:5,pola:2} and \cite{pola:3}, Pola and Tabuada provide a framework to design controllers for systems affected by disturbances.
To this end, they introduce symbolic abstractions for these systems.
Moreover, the notions of approximate simulation \cite{pola:3} and alternating approximate bisimulation \cite{pola:5,pola:2} are introduced to capture the equivalence between symbolic abstractions  and original control systems.

However, as we will reveal in Section~\ref{Sec:logical}, Pola and Tabuada's finite (symbolic) abstractions and their original control systems do not always share the identical properties described by linear temporal logic LTL$_{-X}$.
Roughly speaking, the result ($*$) does not always hold for control systems with disturbances.
Thus, if we adopt the same specifications for the control systems and their finite abstractions, the formal design of the latter may not be helpful for the former.
To overcome this obstacle, this paper introduces and explores a transformation of specification as illustrated in Fig~\ref{fig:illust1}.
\begin{figure}[t]
\begin{center}
\centerline{\includegraphics[scale=0.6]{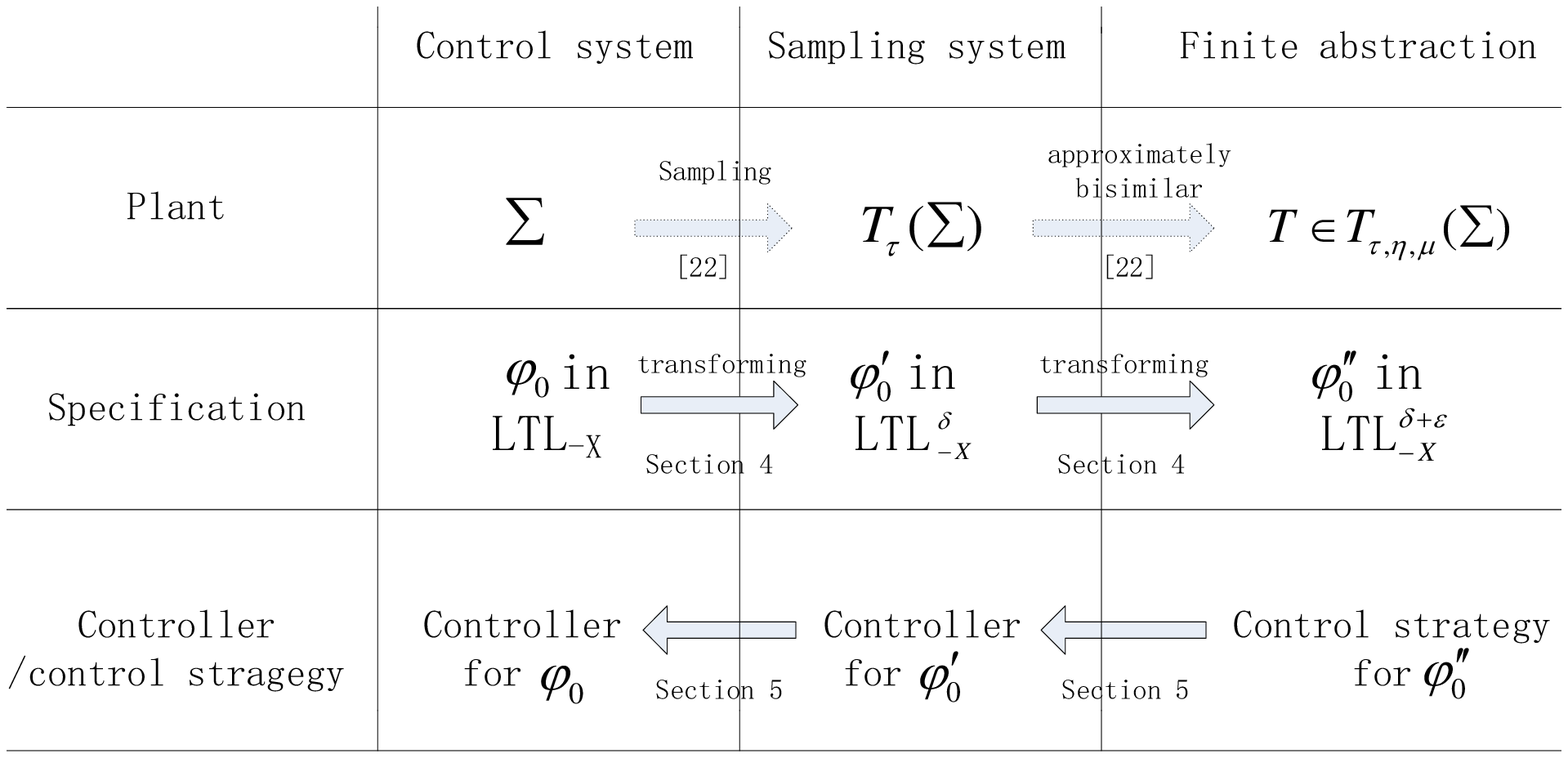}}\vspace{-0.6in}
\end{center}
\caption{Transforming specification and control strategy}\label{fig:illust1}\vspace{-0.3in}
\end{figure}

In this figure, $\Sigma$ is a linear system with disturbance inputs, $T_{\tau}(\Sigma)$ is a sample system of $\Sigma$ and $T_{\tau,\eta,\mu}(\Sigma)$ is the set of finite abstractions of $\Sigma$ introduced in \cite{pola:1}.
Given a linear temporal logic LTL$_{-X}$ formula $\varphi_{0}$ as a specification of $\Sigma$, we transform it to LTL$_{-X}^{\delta}$ formula $\varphi_{0}'$ (LTL$_{-X}^{\delta+\varepsilon}$ formula $\varphi_{0}''$) as specifications of $T_{\tau}(\Sigma)$ (finite abstraction $T$, respectively).
The parametric $\delta$ describes the distinction between the trajectories of $\Sigma$ and their sampling, while finite abstraction $T$ is alternatingly $\varepsilon$-approximately bisimilar to the sampling system $T_{\tau}(\Sigma)$.
It will be shown that, under some assumptions, for any initial state $q_{0}$ and control strategy $f$ of finite abstraction $T$ enforcing $\varphi_{0}''$, there exists a controller of $\Sigma$ derived from $q_{0}$ and $f$ such that the trajectories of $\Sigma$ with this controller satisfy the specification $\varphi_{0}$.

The rest of this paper is organized as follows. In Section~\ref{Sec:pre}, we
recall related definitions and results in the literature.
Section~\ref{Sec:problem} recalls the linear temporal logic LTL$_{-X}$, which is adopted to describe the specifications of linear systems with disturbance inputs.
In Section~\ref{Sec:logical}, we introduce the transformation of LTL$_{-X}$ formulas.
 Based on this transformation, Section~\ref{Sec:construct controller} establishes a relationship between the controller of linear control systems with disturbance inputs and the control strategy of Pola and Tabuada's abstractions.
Finally, we conclude the paper with future work in Section~\ref{Sec:discussion}.
\section{Preliminaries}\label{Sec:pre}
\subsection{Notation}
The symbols $\mathbb{Z}$, $\mathbb{N}$, $\mathbb{R}$, $\mathbb{R}_{+}$ and  $\mathbb{R}_{+}^{0}$ denote the set of integers, positive integers, reals, positive and nonnegative reals, respectively.
Given a function $f:A\rightarrow B$ and $A'\subseteq A$, $f(A')\triangleq\{b\in B: b=f(a)\ \mathrm{for}\ \mathrm{some}\ a\in A'\}$ and the notation $f_{\downarrow A'}$ means the restriction of function $f$ to the set $A'$.
For any set $A$, $A^{+}$ denotes the set of all non-empty finite strings over $A$, and $A^{\omega}$ represents the set of infinite strings over $A$.
We use $s_{A}$ and $\sigma_{A}$ to denote the elements of $A^{+}$ and $A^{\omega}$, respectively. If $A$ is known from the context, we will omit subscripts in $s_{A}$ and $\sigma_{A}$.
For any $s\in A^{+}$, we use $s[i]$ and $s[end]$ to denote the $i$-th element and the last element of $s$, respectively.
Given $i\leq j$, $s[i,j]$, $s[i,end]$ and $\sigma[i,\infty]$ represent $s[i]s[i+1]\cdots s[j]$, $s[i]s[i+1]\cdots s[end]$ and $\sigma[i]\sigma[i+1]\cdots$, respectively. As usual, $|s|$ means the length of $s$. For any $\sigma\in A^{\omega}$, $|\sigma|$ is set to be $\infty$.

Given a vector $\textit{x}\in\mathbb{R}^{n}$,
we denote by $\textit{x}_{i}$ the $\textit{i}$-th element of $\textit{x}$ and $\|x\|\triangleq \max \{|\textit{x}_{1}|, |\textit{x}_{2}|, \cdots, |\textit{x}_{n}|\}$ where $|\textit{x}_{i}|$ is the absolute value of $\textit{x}_{i}$.
For any matrix $\textit{M}\in\mathbb{R}^{n \times m}$, the symbol $\|\textit{M}\|$ represents the infinity norm of $\textit{M}$, i.e., $\| \textit{M} \|\triangleq \max_{1\leq i \leq m}\Sigma^{n}_{j=1}|\textit{a}_{ij}|$.
The set $X\subseteq \mathbb{R}^{n}$ is said to be bounded if and only if $ \sup\{\|x\|:x\in X\}<\infty$.
For any measurable function $f:\mathbb{R}_{+}^{0}\rightarrow \mathbb{R}$, $\|f\|_{\infty}\triangleq\sup\{\|f(t)\|,t\geq 0\}$ and $f$ is said to be essentially bounded if $\|f\|_{\infty}<\infty$.
For a given time $\tau\in\mathbb{R}_{+}$, define $f_{\tau}$ so that $f_{\tau}(t)=f(t)$ for any $t\in[0,\tau)$, and $f(t)=0$ elsewhere; $f$ is said to be locally essentially bounded if for any $\tau\in\mathbb{R}_{+}$, $f_{\tau}$ is essentially bounded.
The symbol $conv(v_{1},v_{2},\cdots, v_{m})$ denotes the convex hull of vectors $v_{1},v_{2},\cdots, v_{m}\in \mathbb{R}^{n}$.
A bounded set of the form $conv(v_{1},v_{2},\cdots, v_{m})$ is called a polytope.
For any $A\subseteq \mathbb{R}^{n}$ and $\mu \in \mathbb{R}$, we define $[\textit{A}]_{\mu}\triangleq\{\textit{x} \in \textit{A} \ |\  x_{i}=k_{i}\mu, k_{i} \in \mathbb{Z}, i=1,\cdots,n\}$. The closed ball centered at $x\in \mathbb{R}^{n}$ with radius $\varepsilon$ is defined by ${\cal B}_{\varepsilon}(x)\triangleq\{y\in \mathbb{R}^{n} \ :  \|x-y\|\leq \varepsilon\}$. In this paper, we consider the metric $\mathbf{d}$ on $\mathbb{R}^{n}$ defined as $\mathbf{d}(x,y)=\|x-y\|$. The Hausdorff pseudo-metric $\mathbf{d}_{h}$ induced by $\mathbf{d}$ on $2^{\mathbb{R}}$ is defined as for any $X_{1},X_{2}\subseteq \mathbb{R}^{n}$,
$\mathbf{d}_{h}(X_{1},X_{2})=\max\{\sup_{x_{1}\in X_{1}}\inf_{x_{2}\in X_{2}} \mathbf{d}(x_{1},x_{2}),\sup_{x_{2}\in X_{2}}\inf_{x_{1}\in X_{1}} \mathbf{d}(x_{1},x_{2})\}.$
\subsection{Linear systems with disturbance inputs}\label{Sec:linear system}
This subsection will recall the notion of linear system with disturbance inputs. We refer the reader to \cite{pola:3,pola:1} for more details. This paper considers the following continuous-time linear control system:
\begin{eqnarray}\label{Eq:system}
\Sigma : \dot{x}=\mathbf{A}x+\mathbf{B}u+\mathbf{G}v,\ \ \  x\in X,u\in U,v\in V
\end{eqnarray}
where $\mathbf{A}\in\mathbb{R}^{n\times n}, \mathbf{B}\in\mathbb{R}^{n\times m}, \mathbf{G}\in\mathbb{R}^{n\times k}$, $X\subseteq \mathbb{R}^{n}$ is the state space, $U\subseteq\mathbb{R}^{m}$ is the control input space, and $V\subseteq\mathbb{R}^{k}$ is the disturbance input space.
We suppose that ${\cal U}$ and ${\cal V}$ are the sets of all measurable and locally essentially bounded functions  from intervals $D\subseteq \mathbb{R}_{+}^{0}$ to $U$ and $V$, respectively, where $D$ is in one of the following forms: $[t_1,t_2]$ and $[t,\theta)$~\footnote{Here, $\theta$ may be equal to $\infty$.}.
For any interval $D\subseteq \mathbb{R}_{+}^{0}$ of the form $[t_1,t_2]$  or $[t,\theta)$, an absolutely continuous curve $\mathbf{x}:D\rightarrow X$ is said to be a trajectory of $\Sigma$ if there exists $\mathbf{u}\in {\cal U}$ and $\mathbf{v}\in {\cal V}$ such that $
\mathbf{\dot{x}}(t_{1})=\mathbf{A}\mathbf{x}(t_{1})+\mathbf{B}\mathbf{u}(t_{1})+\mathbf{G}\mathbf{v}(t_{1})$
for almost all $t_{1}\in D$.
The state reached at time $t\in \mathbb{R}_{+}^{0}$ with initial condition $x_{0}\in X$, control input $\textbf{u}\in {\cal U}$ and disturbance input $\textbf{v}\in{\cal V}$ will be denoted by $\textbf{x}(t,x_{0},\textbf{u},\textbf{v})$. Since $\Sigma$ is a linear system, we have
$\textbf{x}(t,x_{0},\textbf{u},\textbf{v}) =\textbf{x}(t,x_{0},0,0)+\textbf{x}(t,0,\textbf{u},0)+\textbf{x}(t,0,0,\textbf{v})
=e^{\mathbf{A}t}x_{0}+\textbf{x}(t,0,\textbf{u},0)+\textbf{x}(t,0,0,\textbf{v}).$

\textbf{Convention.} \textit{As in \cite{pola:3,pola:1}, we assume that the product $U\times V$ of control input space $U$ and disturbance input space $V$ is compact, and $X\subseteq \mathbb{R}^{n}$ is a bounded polytopic sets with non-empty interior and $0\in X$. Moreover, we assume that the linear control system $\Sigma$ is forward complete and asymptotically stable \footnote{A linear control system is said to be forward complete if and only if for any initial state $x\in X$, control input $\mathbf{u}: \mathbb{R}_{+}^{0} \rightarrow U$ and disturbance input $\mathbf{v}: \mathbb{R}_{+}^{0} \rightarrow V$, there exists a trajectory $\mathbf{x}:\mathbb{R}_{+}^{0}\rightarrow X$ such that $\mathbf{x}(0)=x$ and $\mathbf{\dot{x}}(t)=\mathbf{A}\mathbf{x}(t)+\mathbf{B}\mathbf{u}(t)+\mathbf{G}\mathbf{v}(t)$ for almost all $t\in \mathbb{R}_{+}^{0}$~\cite{angeli}.
The definition of asymptotical stability may be found in~\cite{Kuo,pola:5,pola:3}.}. }

\subsection{Finite abstraction of $\Sigma$}\label{Sec:finite abstraction}
This subsection will recall the construction of finite abstraction of linear system~$\Sigma$ with disturbance inputs, which is introduced by Pola and Tabuada in \cite{pola:1}.
Since inputs consist of control and disturbance inputs, where the former are controllable and the latter are not, usual transition systems can not capture the different roles played by these two kinds of inputs.
To overcome this defect,
Pola and Tabuada adopt alternating transition systems as models of these control systems and their abstract systems \cite{pola:5,pola:2,pola:3}.

\begin{definition}
An alternating transition system is a tuple
$T=(Q,A,B,\longrightarrow,O,H)$
consisting of a set of states $Q$, a set of control labels $A$, a set of disturbance labels $B$, a transition relation $\rightarrow\subseteq Q\times A\times B\times Q$, an observation set $O$ and an observation function $H:Q\rightarrow O$.
We say that an alternating transition system $T$ is metric if the observation set $O$ is equipped with a metric,
$T$ is non-blocking if $\{q':q\xrightarrow{a,b}q'\}\not=\emptyset$ for any $q\in Q$, $a\in A$ and $b\in B$, and
$T$ is finite if $Q$, $A$ and $B$ are finite.
An infinite sequence $\sigma\in Q^{\omega}$ is said to be a trajectory of $T$ if and only if for all $i\in\mathbb{N}$, $\sigma[i]\xrightarrow{a_{i},b_{i}} \sigma[i+1]$ for some $a_{i}\in A $ and $b_{i}\in B$.
\end{definition}

 In the above definition, a transition label is a pair $<a,b>$, where the former is used to denote control input and the latter represents disturbance input.
To obtain a finite abstraction, Pola and Tabuada introduce a notion of sampling system of linear system. In the area of digital control, sampling system has been widely applied as a fundamental notion \cite{Kuo}.

\begin{definition}\label{Def:tau}\cite{pola:5}
Given a linear control system $\Sigma$ below
\begin{eqnarray*}
\Sigma : \dot{x}=\mathbf{A}x+\mathbf{B}u+\mathbf{G}v,\ \ \  x\in X,u\in U,v\in V
\end{eqnarray*}
and $\tau \in \mathbb{R}_{+}$, define the transition system
$T_{\tau}(\Sigma)=(Q_{\tau}, A_{\tau}, B_{\tau}, \xrightarrow[\tau]{} , O_{\tau}, H_{\tau})$,
where:

$\bullet$ $Q_{\tau}=X$;

$\bullet$ $A_{\tau}=\{\mathbf{u}\in {\cal U}: \ the\ domain\ of\ \mathbf{u}\ is\ [0,\tau]\}$;

$\bullet$ $B_{\tau}=\{\mathbf{v}\in {\cal V}:\ the\ domain\ of\ \mathbf{v}\ is\ [0,\tau]\}$;

$\bullet$ $q\xrightarrow[\tau]{\mathbf{u},\mathbf{v}} q'$ if $\mathbf{x}(\tau, q, \mathbf{u},\mathbf{v})=q'$;

$\bullet$ $O_{\tau}=X$;

$\bullet$ $H_{\tau}=1_{X}$ is the identity map on the set $X$.

\end{definition}

Let $\mathbf{x}:\mathbb{R}_{+}^{0}\rightarrow X$ be a trajectory of $\Sigma$. Given $\tau\in \mathbb{R}_{+}$, we set $\sigma_{{\mathbf{x}}}^{\tau}=\mathbf{x}(0)\mathbf{x}(\tau)\mathbf{x}(2\tau)\cdots$.
The sequence $\sigma_{{\mathbf{x}}}^{\tau}$ can be viewed as a sampling of $\mathbf{x}$.
It is easy to check that $\sigma_{{\mathbf{x}}}^{\tau}$ is a trajectory of $T_{\tau}(\Sigma)$.
For simplicity, if $\tau$ is known from the context, we often omit the superscript in $\sigma_{{\mathbf{x}}}^{\tau}$.
In order to extract a finite abstraction from $T_{\tau}(\Sigma)$, the following notations are needed:
\begin{equation*}
\begin{aligned}
\mathcal{R}_{A_{\tau}}& \triangleq\{q\in \mathbb{R}^{n}:0\xrightarrow[\tau]{\mathbf{u},0} q \textrm{ for some }\mathbf{u}\in A_{\tau}\},\textrm{ and }\\
\mathcal{R}_{B_{\tau}}& \triangleq\{q\in \mathbb{R}^{n}:0\xrightarrow[\tau]{0,\mathbf{v}} q \textrm{ for some }\mathbf{v}\in A_{\tau}\}.
\end{aligned}
\end{equation*}
It is easy to see that $\mathcal{R}_{A_{\tau}}$  is the set of all reachable states from the initial state 0 with some  control input $\mathbf{u}$ and identically null disturbance input $\mathbf{0}$.
Similarly, $\mathcal{R}_{B_{\tau}}$ is the set of states reached at time $\tau$ from the initial state $0$ with control input $\mathbf{0}$ and some disturbance input $\mathbf{v}$.
The computation of these sets can be found in~\cite{pola:1}.
The notion of an abstract model for $\Sigma$ is recalled below.

\begin{definition}\label{Def:fin}\cite{pola:1}
Given a linear control system $\Sigma$ below
\begin{eqnarray*}
\Sigma : \dot{x}=\mathbf{A}x+\mathbf{B}u+\mathbf{G}v,\ \ \  x\in X,u\in U,v\in V
\end{eqnarray*}
and $\tau,\eta,\mu \in \mathbb{R}_{+}$, an alternating transition system $T=([X]_{\eta},A, B, \rightarrow, \mathbb{R}^{n}, H)$ is said to be an abstraction of $\Sigma$ w.r.t $\tau, \eta$ and $\mu$ if and only if it satisfies:

(1) $A\subseteq[\mathbb{R}^{n}]_{\mu}$ and $\mathbf{d}_{h}(A,\mathcal{R}_{A_{\tau}})\leq \mu/2$;

(2) $B\subseteq[\mathbb{R}^{n}]_{\mu}$ and $\mathbf{d}_{h}(B,\mathcal{R}_{B_{\tau}})\leq \mu/2$;

(3) $q\xrightarrow{a,b} q'$ if and only if
$\|\mathbf{x}(\tau, q, 0,0)+a+b-q'\|\leq\eta/2$;

(4) $H:[X]_{\eta}\hookrightarrow \mathbb{R}^{n}$ is a natural inclusion map.

We set $T_{\tau, \eta, \mu}(\Sigma)\triangleq\{T:T\textrm{ is a finite abstraction of } \Sigma \textrm{ w.r.t. }\tau, \eta\textrm{ and }\mu\}$.
\end{definition}

Since we have supposed that the linear system $\Sigma$ is forward complete, the sample system $T_{\tau}$ and any abstraction  of $\Sigma$ are non-blocking~\cite{pola:1}.
Moreover, for any $\tau,\eta,\mu \in \mathbb{R}_{+}$, the boundedness of the state space $X$ of $\Sigma$ implies that any abstraction of $\Sigma$ w.r.t $\tau,\eta$ and $\mu$ is finite~\cite{pola:1}.
In order to capture the equivalence between the finite abstraction and the sampling system of the original linear system, Pola and Tabuada introduce the notion of alternating approximate bisimulation.

\begin{definition}\label{Def:bis} \cite{pola:5,pola:2}
Let $T_{i}=(Q_{i},A_{i}, B_{i}, \xrightarrow[i]{}, O, H_{i})$ ($i=1,2$) be two metric, non-blocking alternating transition systems with the same observation set and the same metric~$\textbf{d}$ over $O$. Given a precision $\varepsilon \in \mathbb{R}_{+}^{0}$, a relation $R\subseteq Q_{1} \times Q_{2}$ is said to be an alternating $\varepsilon$-approximate ($A\varepsilon A$) bisimulation relation between $T_{1}$ and $T_{2}$ if for any $(q_{1},q_{2})\in R$,

(i) $\textbf{d}(H_{1}(q_{1}),H_{2}(q_{2}))\leq\varepsilon$;

(ii) $\forall a_{1}\in A_{1} \exists a_{2}\in A_{2} \forall b_{2} \in B_{2} \forall q'_{2}\in Q_{2} (q_{2}\xrightarrow[2]{a_{2},b_{2}} q'_{2} \Rightarrow \exists b_{1}\in B_{1} {\exists q'_{1}\in Q_{1}} ({q_{1}\xrightarrow[1]{a_{1},b_{1}} q'_{1}} \textrm{ and } \\ (q'_{1},q'_{2})\in R))$.

(iii) $\forall a_{2}\in A_{2} \exists a_{1}\in A_{1} \forall b_{1} \in B_{1} \forall q'_{1}\in Q_{1} (q_{1}\xrightarrow[1]{a_{1},b_{1}} q'_{1} \Rightarrow \exists b_{2}\in B_{2} \exists q'_{2}\in Q_{2}  ({q_{2}\xrightarrow[2]{a_{2},b_{2}} q'_{2}} \textrm{ and } \\ (q'_{1},q'_{2})\in R))
$.

For any $q_{1}\in Q_{1}$ and $q_{2}\in Q_{2}$, they are said to be $A\varepsilon A$ bisimilar, in symbols $q_{1}\sim_{\varepsilon} q_{2}$,  if there exists an $A\varepsilon A$ bisimulation relation $R$ between $T_{1}$ and $T_{2}$ such that $(q_{1},q_{2})\in R$.
Moreover, $T_{1}$ and $T_{2}$ are said to be $A\varepsilon A$ bisimilar, in symbols $T_{1}\simeq_{\varepsilon} T_{2}$,  if there exists an $A\varepsilon A$ bisimulation relation $R$ between $T_{1}$ and $T_{2}$ such that $Q_{1}=\{q_{1}\in Q_{1}: (q_{1},q_{2})\in R\ \mathrm{for}\ \mathrm{some}\ q_{2}\in Q_{2}\}$ and $Q_{2}=\{q_{2}\in Q_{2}: (q_{1},q_{2})\in R\ \mathrm{for}\ \mathrm{some}\ q_{1}\in Q_{1}\}$.
\end{definition}

Immediately, we have the following result as usual. We leave its proof to the interested reader. Similar proofs may be found in \cite{milner,Ying}.

\begin{proposition}\label{Pro:approximate bisimulation}
$q_{1}\sim_{\varepsilon} q_{2}$ if and only if they satisfy the following conditions:

(i) $\textbf{d}(H_{1}(q_{1}),H_{2}(q_{2}))\leq\varepsilon$;

(ii) $\forall a_{1}\in A_{1} \exists a_{2}\in A_{2} \forall b_{2} \in B_{2} \forall q'_{2}\in Q_{2} (q_{2}\xrightarrow[2]{a_{2},b_{2}} q'_{2} \Rightarrow \exists b_{1}\in B_{1} \exists q'_{1}\in Q_{1} ({q_{1}\xrightarrow[1]{a_{1},b_{1}} q'_{1}} \textrm{ and } \\ q'_{1}\sim_{\varepsilon} q'_{2}))$.

(iii) $\forall a_{2}\in A_{2} \exists a_{1}\in A_{1} \forall b_{1} \in B_{1} \forall q'_{1}\in Q_{1}
(q_{1}\xrightarrow[1]{a_{1},b_{1}} q'_{1} \Rightarrow \exists b_{2}\in B_{2} \exists q'_{2}\in Q_{2} ({q_{2}\xrightarrow[2]{a_{2},b_{2}} q'_{2}} \textrm{ and } \\ q'_{1}\sim_{\varepsilon} q'_{2}))
$.
\end{proposition}

Under some circumstances, the sampling system $T_{\tau}(\Sigma)$  and finite abstraction of a control system $\Sigma$ are shown to be alternatingly approximately bisimilar.

\begin{theorem}\cite{pola:1}\label{Th:bisi}
Given an asymptotically stable linear control system $\Sigma$ below
\begin{eqnarray*}
\Sigma : \dot{x}=\mathbf{A}x+\mathbf{B}u+\mathbf{G}v,\ \ \  x\in X,u\in U,v\in V
\end{eqnarray*}
and $\varepsilon\in\mathbb{R}_{+}$. For any $\tau, \eta, \mu\in\mathbb{R}_{+}$ satisfying
$\|e^{\mathbf{A}\tau}\|\varepsilon +\mu+\eta/2 < \varepsilon$
and for any finite abstraction $T\in T_{\tau, \eta, \mu}(\Sigma)$, $T$ is $A\varepsilon A$ bisimilar to $T_{\tau}(\Sigma)$ and for any state $q_{1}$ of $T$ and state $q_{2}$ of $T_{\tau}$, if $\mathbf{d}(q_{1},q_{2})\leq \varepsilon$ then $q_{1}\sim_{\varepsilon}q_{2}$.
\end{theorem}
\section{Linear temporal logic LTL$_{-X}$}\label{Sec:problem}
The notion of alternating transition system provides a formal model for control system with disturbance inputs.
Apart from formal model, formal specification is another basic element in the formal analysis and design of control systems.
The former captures the dynamics of control system, while the latter describes the desired property that control system should satisfy.
As mentioned in Introduction, temporal logic is widely adopted to describe task specification~\cite{alur,anton:1,fain:1,fain:2,kloe:1,tab:1}.
In this paper, the specification of $\Sigma$ will be expressed by a linear temporal logic known as LTL$_{-X}$~\cite{emerson}.
The LTL$_{-X}$ formulae have been used to specify the desired properties of control systems in~\cite{kloe:1}. We recall this logic below.
\subsection{LTL$_{-X}$ and satisfaction relation in discrete case}
Given a finite set $\mathbb{P}$ of atomic propositions, the temporal logic LTL$_{-X}(\mathbb{P})$ is defined as follows.

\begin{definition}\cite{kloe:1,emerson}
Let $\mathbb{P}$ be a finite set of atomic propositions. The linear temporal logic LTL$_{-X}(\mathbb{P})$ formula is inductively defined as:
\begin{center}
$\varphi::=p|\neg p|\varphi_{1}\wedge\varphi_{2}|\varphi_{1}\vee\varphi_{2}|\varphi_{1}\mathbf{U}\varphi_{2}|\varphi_{1}\tilde{\mathbf{U}}\varphi_{2}$
\end{center}
where $p\in \mathbb{P}$.
\end{definition}

The operator $\mathbf{U}$ is read as ``until'' and the formula $\varphi_{1}\mathbf{U}\varphi_{2}$ specifies that $\varphi_{1}$ must hold until $\varphi_{2}$ holds.
The operator $\tilde{\mathbf{U}}$ is the dual of $\mathbf{U}$ and is best read as ``releases''.
The semantics of LTL$_{-X}(\mathbb{P})$ formulae are defined below.

\begin{definition}\label{Def:satisfaction}
Let $\sigma_{\mathbb{P}}$ be any infinite word over $2^{\mathbb{P}}$ (i.e., $\sigma_{\mathbb{P}}\in (2^{\mathbb{P}})^{\omega}$).
 The satisfaction of LTL$_{-X}(\mathbb{P})$ formula $\varphi$ at position $i\in \mathbb{N}$ of word $\sigma_{\mathbb{P}}$, denoted by $\sigma_{\mathbb{P}}[i]\models\varphi$, is defined inductively as follows:

(1) $\sigma_{\mathbb{P}}[i]\models p$ iff $p\in \sigma_{\mathbb{P}}[i]$;

(2) $\sigma_{\mathbb{P}}[i]\models \neg p$ iff $p\notin \sigma_{\mathbb{P}}[i]$;

(3) $\sigma_{\mathbb{P}}[i]\models\varphi_{1}\wedge\varphi_{2}$ iff $\sigma_{\mathbb{P}}[i]\models\varphi_{1}$ and $\sigma_{\mathbb{P}}[i]\models\varphi_{2}$;

(4) $\sigma_{\mathbb{P}}[i]\models\varphi_{1}\vee\varphi_{2}$ iff $\sigma_{\mathbb{P}}[i]\models\varphi_{1}$ or $\sigma_{\mathbb{P}}[i]\models\varphi_{2}$;

(5) $\sigma_{\mathbb{P}}[i]\models\varphi_{1}\mathbf{U}\varphi_{2}$ iff there exists $j\geq i$ such that $\sigma_{\mathbb{P}}[j]\models\varphi_{2}$ and for all $k\in\mathbb{N}$ with $i\leq k<j$, we have $\sigma_{\mathbb{P}}[k]\models\varphi_{1}$;

(6) $\sigma_{\mathbb{P}}[i]\models\varphi_{1}\tilde{\mathbf{U}}\varphi_{2}$ iff for all $j\geq i$ with $\sigma_{\mathbb{P}}[j]\not\models\varphi_{2}$, there exists $k\in\mathbb{N}$ such that $i\leq k<j$ and $\sigma_{\mathbb{P}}[k]\models\varphi_{1}$.

An infinite word $\sigma_{\mathbb{P}}$ is said to satisfy an LTL$_{-X}(\mathbb{P})$ formula $\varphi$, written as $\sigma_{\mathbb{P}}\models\varphi$, if and only if $\sigma_{\mathbb{P}}[1]\models \varphi$.
\end{definition}

\begin{definition}\label{Def:trajectory to word 1}
Let $\mathbb{P}$ be a finite set of atomic propositions and let $\prod:\mathbb{R}^{n}\rightarrow 2^{\mathbb{P}}$ be a valuation function.
Then for any LTL$_{-X}(\mathbb{P})$ formula $\varphi$, an infinite sequence $\sigma\in(\mathbb{R}^{n})^{\omega}$ is said to satisfy $\varphi$  w.r.t $\prod$, written as $\sigma\models_{\prod}\varphi$, if and only if $\prod(\sigma)\models\varphi$, where $\prod(\sigma)\triangleq\prod(\sigma[1])\prod(\sigma[2])\cdots$.
\end{definition}

In this paper, similar to \cite{kloe:1}, we fix a finite set $\mathbb{P}\!_h$ of atomic propositions, where each proposition $p\in\mathbb{P}\!_h$ denotes an open half-space  of $\mathbb{R}^{n}$, i.e.,
$p=\{x\in \mathbb{R}^{n}:c^{T}_{p}x +d_{p}<0\}$
with $c_{p}\in\mathbb{R}^{n}$ and $d_{p}\in \mathbb{R}$.
So the valuation function $\prod_h$ considered in this paper is defined as: for any $q\in\mathbb{R}^{n}$, $\prod_h(q)\triangleq\{p\in\mathbb{P}\!_h:q\in p\}$.
Henceforth, since $\mathbb{P}\!_h$ and $\prod_h$ are fixed, we will abbreviate LTL$_{-X}(\mathbb{P}\!_h)$ to LTL$_{-X}$ and omit the subscript in $\models_{\prod_h}$.
\subsection{Satisfaction relation in continuous case}
This subsection will explore the satisfaction relation between continuous trajectories of linear system $\Sigma$ and LTL$_{-X}$ formulas.
Kloetzer and Belta have defined such a satisfaction relation based on the notion of word corresponding to continuous trajectory~\cite{kloe:1}.
We will recall their definition.
Moreover, we will provide an alternative definition of satisfaction relation without reference to word.
It will be shown that the latter is coincided with Kloetzer and Belta's.
For simplifying related proofs, the latter will be adopted in the remainder of this paper.
\subsubsection{Satisfaction relation based on word}\
In~\cite{kloe:1}, to define the satisfaction relation between continuous trajectories and LTL$_{-X}$ formulas, the notion of word corresponding to continuous trajectory is introduced.

\begin{definition}\cite{kloe:1}\label{Def:infinite word}
Let $\Sigma$ be a linear control system with state space $X$ and $\mathbf{x}:\mathbb{R}_{+}^{0}\rightarrow X$ a trajectory of $\Sigma$.
An infinite sequence $\sigma\in (2^{\mathbb{P}\!_h})^{\omega}$ is said to be the word corresponding to the trajectory $\mathbf{x}$ if and only if there exist $t_{i}\in \mathbb{R}_{+}^{0} (i\in\mathbb{N})$ with $0=t_{1}<t_{2}<t_{3}<\cdots$ such that for each $i\in \mathbb{N}$,

(1$_{i}$) $\sigma[i]=\prod_h(\mathbf{x}(t_{i}))$;

(2$_{i}$) if $\sigma[i]\neq \sigma[i+1]$ then there exists $t\in[t_{i},t_{i+1}]$ such that one of the following holds:

\ \ \ \ ($2_{i}$-a) $\sigma[i]=\prod_h ({\mathbf{x}(t')})$ and $\sigma[i+1]=\prod_h({\mathbf{x}(t'')})$ for all $t'\in [t_{i},t)$ and~${t''\in [t,t_{i+1}]}$;

\ \ \ \ ($2_{i}$-b) $\sigma[i]=\prod_h({\mathbf{x}(t')})$ and $\sigma[i+1]=\prod_h({\mathbf{x}(t'')})$ for all $t'\in [t_{i},t]$ and $t''\in (t,t_{i+1}]$;

(3$_{i}$) if $\sigma[i]=\sigma[i+1]$ then $\sigma[i]=\prod_h({\mathbf{x}(t)})$ for all $t\in [t_{i},\infty)$.
\end{definition}

\begin{definition}\cite{kloe:1}\label{Def:satisfaction of x}
Let $\Sigma$ be a linear control system with state space $X$, $\mathbf{x}:\mathbb{R}_{+}^{0}\rightarrow X$ a trajectory of $\Sigma$, and let $\varphi$ be an LTL$_{-X}$ formula. The trajectory $\mathbf{x}$ is said to satisfy $\varphi$, written as $\mathbf{x}\models_{w}\varphi$,
if and only if its corresponding word satisfies $\varphi$.
\end{definition}

Clearly, given a trajectory $\mathbf{x}$, whether the above definition is well-defined depends on the existence and uniqueness of the corresponding word of  $\mathbf{x}$.
We will show that, in practical circumstance, this definition works well.
To this end, we introduce the following notion.

\begin{definition}\label{tipping}
Let $\Sigma$ be a linear control system with state space $X$, $\mathbf{x}: \mathbb{R}_{+}^{0}\rightarrow X$ a trajectory of $\Sigma$ and $t\in \mathbb{R}_{+}^{0}$.
Then $t$ is said to be a tipping point of $\mathbf{x}$ w.r.t. $\mathbb{P}\!_h$ if and only if  for any $\varepsilon_{0}\in \mathbb{R}_{+}$, there exists $\varepsilon_{1}<\varepsilon_{0}$ such that
$\prod_h(\mathbf{x}(t-\varepsilon_{1}))\neq \prod_h(\mathbf{x}(t))$ or $\prod_h(\mathbf{x}(t))\neq \prod_h(\mathbf{x}(t+\varepsilon_{1}))$.
For any $t_{0}\in\mathbb{R}_{+}^{0}$, $Tip(t_{0},\mathbf{x})\triangleq\{ t'\in\mathbb{R}_{+}^{0}: t'\leq t_{0}\textrm{ and } t'\textrm{ is a tipping point of } \mathbf{x} \textrm{ w.r.t. }\mathbb{P}\!_h\}$.
\end{definition}

Intuitively, if $t$ is a tipping point of $\mathbf{x}$ w.r.t.~$\mathbb{P}\!_h$, it means that the trajectory $\mathbf{x}$ cuts across a borderline $\{x\in \mathbb{R}^{n}:c_{p}^{T}x+d_{p}=0\}$ for some $p\in \mathbb{P}\!_h$ at time $t$.
Clearly, given a trajectory $\mathbf{x}$ and $t_{1}<t_{2}$, since $\mathbf{x}$ is continuous, if $\prod_h(\mathbf{x}(t_{1}))\neq \prod_h(\mathbf{x}(t_{2}))$ then there exists at least one tipping point $t$ w.r.t.~$\mathbb{P}\!_h$ so that $t\in [t_{1},t_{2}]$.
We leave its proof to interested reader.
The following result explores the existence and uniqueness of the word corresponding to continuous trajectory.
According to this result, if the trajectory $\mathbf{x}$ does not cut across borderlines infinite times on any bounded time interval $[0,t]$, then  Definition~\ref{Def:satisfaction of x} is well-defined for $\mathbf{x}$.

\begin{proposition}\label{exist}
Let $\Sigma$ be a linear control system with state space $X$ and let $\mathbf{x}: \mathbb{R}_{+}^{0}\rightarrow X$ be a trajectory of $\Sigma$. Then the following conclusions hold:

(1) The word corresponding to the trajectory $\mathbf{x}$ is unique if it exists.

(2) If  $Tip(t,\mathbf{x})$
is finite for any $t\in \mathbb{R}_{+}^{0}$, then there exists a word corresponding to~$\mathbf{x}$.
\end{proposition}
\begin{proof}
(1) Suppose that $\sigma_{1}$ and $\sigma_{2}$ are words corresponding to $\mathbf{x}$.
Then for $n=1,2$, by Definition~\ref{Def:infinite word}, there exist $t^{n}_{i}\in\mathbb{R}_{+}^{0}$ $(i\in\mathbb{N})$ with $0=t^{n}_{1}<t^{n}_{2}<\cdots$ such that for any $i\in\mathbb{N}$, ($1_{i}$), ($2_{i}$) and ($3_{i}$) in Definition~\ref{Def:infinite word} hold for $\sigma_{n}$ and $\mathbf{x}$.
To prove $\sigma_{1}=\sigma_{2}$, it suffices to show that $\sigma_{1}[i]=\sigma_{2}[i]=\prod_h(\mathbf{x}(t))$ for any $i\in\mathbb{N}$ and $t\in\mathbb{R}_{+}^{0}$ with $t^{1}_{i}\leq t\leq t^{2}_{i}$ or $t^{2}_{i}\leq t\leq t^{1}_{i}$.
We argue by induction on $i$.

If $i=1$ then $t^{1}_{i}=t^{2}_{i}=0$ and the conclusion holds trivially.

Suppose that the conclusion holds for $k$ and $i=k+1$. Consider two cases below.

 \textbf{Case 1}. $\sigma_{1}[k]=\sigma_{1}[k+1]$ or $\sigma_{2}[k]=\sigma_{2}[k+1]$.

Suppose that $\sigma_{1}[k]=\sigma_{1}[k+1]$.
Then, by Definition~\ref{Def:infinite word}, we have
\begin{equation}\label{eq:3.1}
\textrm{$\prod$}_h(\mathbf{x}(t))=\sigma_{1}[k] \textrm{ for any }t>t^{1}_{k}.
\end{equation}
Moreover, by induction hypothesis, we obtain $\sigma_{1}[k]=\sigma_{2}[k]=\prod_h(\mathbf{x}(t))$ for any $t\in\mathbb{R}_{+}^{0}$ with $t^{1}_{k}\leq t\leq t^{2}_{k}$ or $t^{2}_{k}\leq t\leq t^{1}_{k}$.
Thus, it follows from (\ref{eq:3.1}) that $\sigma_{1}[k]=\sigma_{2}[k]=\prod_h(\mathbf{x}(t))$ for any $t>t^{2}_{k}$.
Therefore, since $\sigma_{n}[k+1]=\prod_h(\mathbf{x}(t_{k+1}^{n}))$ and $t_{k+1}^{n}>t_{k}^{n}$ for $n=1,2$, we get $\sigma_{1}[k]=\sigma_{1}[k+1]=\sigma_{2}[k+1]=\prod_h(\mathbf{x}(t))$ for any $t\in\mathbb{R}_{+}^{0}$ with $t^{1}_{k+1}\leq t\leq t^{2}_{k+1}$ or $t^{2}_{k+1}\leq t\leq t^{1}_{k+1}$.

Similarly, if $\sigma_{2}[k]=\sigma_{2}[k+1]$, we may show that the conclusion holds for $k+1$.

 \textbf{Case 2}. $\sigma_{1}[k]\neq\sigma_{1}[k+1]$ and $\sigma_{2}[k]\neq\sigma_{2}[k+1]$.

If $t_{k+1}^{1}=t_{k+1}^{2}$ then the conclusion holds for $k+1$ trivially.
So we just need to consider the nontrivial case where $t_{k+1}^{1}\neq t_{k+1}^{2}$.
Without loss of generality, we may assume that $t_{k+1}^{1}<t_{k+1}^{2}$.
By induction hypothesis, we have $\sigma_{1}[k]=\sigma_{2}[k]=\prod_h(\mathbf{x}(t))$ for any $t\in\mathbb{R}_{+}^{0}$ with $t^{1}_{k}\leq t\leq t^{2}_{k}$ or $t^{2}_{k}\leq t\leq t^{1}_{k}$.
Then since $\sigma_{1}[k]\neq\sigma_{1}[k+1]$, we obtain  $t_{k}^{2}<t_{k+1}^{1}$ (otherwise,  $\sigma_{1}[k]=\sigma_{1}[k+1]$ follows from $t^{1}_{k}< t_{k+1}^{1}\leq t^{2}_{k}$ and induction hypothesis).
Furthermore, by $\sigma_{2}[k]\neq\sigma_{2}[k+1]$ and Definition~\ref{Def:infinite word}, there exists $t\in[t_{k}^{2},t_{k+1}^{2}]$ such that one of the following holds:

($a$) $\sigma_{2}[i]=\prod_h ({\mathbf{x}(t')})$ and $\sigma_{2}[i+1]=\prod_h({\mathbf{x}(t'')})$ for all $t'\in [t_{k}^{2},t)$ and $t''\in [t,t_{k+1}^{2}]$;

($b$) $\sigma_{2}[i]=\prod_h({\mathbf{x}(t')})$ and $\sigma_{2}[i+1]=\prod_h({\mathbf{x}(t'')})$ for all $t'\in [t_{k}^{2},t]$ and $t''\in (t,t_{k+1}^{2}]$.

Then since $\sigma_{1}[k+1]\neq\sigma_{2}[k]$ and $t_{k}^{2}<t_{k+1}^{1}<t_{k+1}^{2}$, we get
$\sigma_{1}[k+1]=\sigma_{2}[k+1]$ and $t\leq t_{k+1}^{1}$.
Further, it follows that $\sigma_{1}[k+1]=\sigma_{2}[k+1]=\prod_h(\mathbf{x}(t'))$ for any $t'\in[t_{k+1}^{1},t_{k+1}^{2}]$.

(2)
Suppose that $Tip(t,\mathbf{x})$ is finite for  any $t\in \mathbb{R}_{+}^{0}$.
By Definition~\ref{Def:infinite word}, it is enough to construct infinite sequences $t_{1}t_{2}\cdots  \in (\mathbb{R}_{+}^{0})^{\omega}$ and $\sigma \in (2^{\mathbb{P}\!_h})^{\omega}$ so that $0= t_{1}<t_{2}<\cdots$ and for any $i\in\mathbb{N}$, ($1_{i}$), ($2_{i}$) and ($3_{i}$) in Definition~\ref{Def:infinite word} hold for $\sigma$ and $\mathbf{x}$. We construct them by induction on $i\in\mathbb{N}$.

We set $t_{1}=0$ and $\sigma[1]=\prod_h(\mathbf{x}(t_{1}))$.

Assuming that we already have $t_{k}$ and $\sigma[k]$, we construct $t_{k+1}$ and $\sigma[k+1]$ below.
If $\prod_h(\mathbf{x}(t_{a}))=\prod_h(\mathbf{x}(t_{a}')$ for all $t_{a},t_{a}'\in(t_{k},\infty)$, then we set $t_{k+1}$ to be an arbitrary real number such that $t_{k+1}> t_{k}$ and put $\sigma[k+1]=\prod_h(\mathbf{x}(t_{k+1}))$.
In the following, we consider the case where $\prod_h(\mathbf{x}(t_{a}))\neq\prod_h(\mathbf{x}(t_{a}')$ for some $t_{a},t_{a}'\in(t_{k},\infty)$ with $t_{a}<t_{a}'$.
Then there exists at least one tipping point $t$ with $t_{a}\leq t\leq t_{a}'$.
Since $Tip(t_{a}',\mathbf{x})$ is finite, there exists $t'\in Tip(t_{a}',\mathbf{x})$ such that
$t'> t_{k}$ and $t''\not\in Tip(t_{a}',\mathbf{x})$ for all $t''\in (t_{k},t')$.
Thus by Definition~\ref{tipping}, one of the following holds:

($i$) for any $\varepsilon_{0}\in \mathbb{R}_{+}^{0}$, there exists $\varepsilon_{1}<\varepsilon_{0}$ such that $\prod_h(\mathbf{x}(t'-\varepsilon_{1}))\neq \prod_h(\mathbf{x}(t'))$,

($ii$) for any $\varepsilon_{0}\in \mathbb{R}_{+}^{0}$, there exists $\varepsilon_{1}<\varepsilon_{0}$ such that $\prod_h(\mathbf{x}(t'))\neq \prod_h(\mathbf{x}(t'+\varepsilon_{1}))$.

If ($i$) holds then we set $t_{k+1}=t'$ and $\sigma[k+1]=\prod_h(\mathbf{x}(t'))$.
Otherwise, ($ii$) holds. Since $Tip(t,\mathbf{x})$ is finite for  any $t\in \mathbb{R}_{+}^{0}$, there exists $\varepsilon_{0}\in \mathbb{R}_{+}^{0}$ such that $Tip(t'+\varepsilon_{0},\mathbf{x})-Tip(t',\mathbf{x})=\emptyset$. We set $t_{k+1}=t'+\varepsilon_{0}$ and $\sigma[k+1]=\prod_h(\mathbf{x}(t_{k+1}))$.

By Definition~\ref{Def:infinite word} and~\ref{tipping}, one may easily check that $\sigma$ defined above is the word corresponding to~$\mathbf{x}$.
\end{proof}

 \textbf{Remark 1}:
In practice, we can not observe that a trajectory cuts across borderlines infinite times on some bounded time interval $[t_{1},t_{2}]$.
So in this paper, we assume that for any trajectory $\mathbf{x}$ of $\Sigma$ and $t\in\mathbb{R}_{+}^{0}$, $Tip(t,\mathbf{x})$ is finite.
Then by Proposition~\ref{exist}, Definition~\ref{Def:satisfaction of x} is well-defined.
\subsubsection{Satisfaction relation based on trajectory}
In this subsection, we will define the satisfaction relation between continuous trajectories and LTL$_{-X}$ formulas without reference to word.
This satisfaction relation will be shown to be coincided with the one in Definition~\ref{Def:satisfaction of x}.

\begin{definition}\label{Def:satisfaction of x'}
Let $\Sigma$ be a linear control system with state space $X$ and let $\mathbf{x}:\mathbb{R}_{+}^{0}\rightarrow X$ be a trajectory of $\Sigma$. The satisfaction of LTL$_{-X}$ formula $\varphi$ at time $t\in \mathbb{R}_{+}^{0}$ of $\mathbf{x}$, denoted by $\mathbf{x}(t)\models\varphi$, is defined inductively as:

(1) $\mathbf{x}(t)\models p$ iff $\mathbf{x}(t)\in p$;

(2) $\mathbf{x}(t)\models \neg p$ iff $\mathbf{x}(t)\not\in p$;

(3) $\mathbf{x}(t)\models\varphi_{1}\wedge\varphi_{2}$ iff $\mathbf{x}(t)\models\varphi_{1}$ and $\mathbf{x}(t)\models\varphi_{2}$;

(4) $\mathbf{x}(t)\models\varphi_{1}\vee\varphi_{2}$ iff $\mathbf{x}(t)\models\varphi_{1}$ or $\mathbf{x}(t)\models\varphi_{2}$;

(5) $\mathbf{x}(t)\models\varphi_{1}\mathbf{U}\varphi_{2}$ iff  for some $t_{1},t_{2}\in \mathbb{R}_{+}^{0}$ with $t\leq t_{1}< t_{2}$, one of the following holds:

\ \ \ \ (5-a) $\mathbf{x}(t_{1})\models\varphi_{2}$ and $\mathbf{x}(t')\models\varphi_{1}$ for all $t'\in[t,t_{1})$,

\ \ \ \ (5-b) $\mathbf{x}(t')\models\varphi_{1}$ and $\mathbf{x}(t'')\models\varphi_{2}$ for all $t'\in[t,t_{1}]$ and  $t''\in(t_{1},t_{2}]$;

(6) $\mathbf{x}(t)\models\varphi_{1}\tilde{\mathbf{U}}\varphi_{2}$ iff for any $t_{1},t_{2}\in \mathbb{R}_{+}^{0}$ with $t\leq t_{1}< t_{2}$, we have

\ \ \ \ (6-a) if $\mathbf{x}(t_{1})\not\models\varphi_{2}$ then $\mathbf{x}(t')\models\varphi_{1}$ for some $t'\in[t,t_{1})$,

\ \ \ \ (6-b) if $\mathbf{x}(t')\not\models\varphi_{2}$ for all $t'\in(t_{1},t_{2}]$ then $\mathbf{x}(t'')\models\varphi_{1}$ for some $t''\in[t,t_{1}]$.

An LTL$_{-X}$ formula $\varphi$ is said to be satisfied by $\mathbf{x}$, written as $\mathbf{x}\models\varphi$, if and only if $\mathbf{x}(0)\models \varphi$.
\end{definition}

In the following, we want to show that for any trajectory $\mathbf{x}$ of $\Sigma$, if $Tip(t,\mathbf{x})$ is finite for all $t\in\mathbb{R}_{+}^{0}$, then for any LTL$_{-X}$ formula $\varphi$, $\mathbf{x}\models\varphi$ if and only if $\mathbf{x}\models_{w}\varphi$.
Before demonstrating it, we introduce a notation and provide an auxiliary lemma.

\textbf{Notation}:
\textit{Let $\Sigma$ be a linear control system with state space $X$, $\mathbf{x}:\mathbb{R}_{+}^{0}\rightarrow X$ a trajectory of $\Sigma$ and $t\in \mathbb{R}_{+}^{0}$. The function $\mathbf{x}^{t}:\mathbb{R}_{+}^{0}\rightarrow X$ is defined as $\mathbf{x}^{t}(t')=\mathbf{x}(t+t')$ for all $t'\in \mathbb{R}_{+}^{0}$.}

Clearly, $\mathbf{x}^{0}=\mathbf{x}$ and $\mathbf{x}^{t}$ is also a trajectory of $\Sigma$ for any $t\in \mathbb{R}_{+}^{0}$. Moreover, by Definition~\ref{Def:satisfaction of x'}, it is easy to check that for any $t\in \mathbb{R}_{+}^{0}$ and LTL$_{-X}$ formula $\varphi$, $\mathbf{x}^{t}\models\varphi$ if and only if $\mathbf{x}(t)\models\varphi$.

\begin{lemma}\label{Lem:subtrajectory and sub word}
Let $\Sigma$ be a linear control system with state space $X$ and let $\mathbf{x}:\mathbb{R}_{+}^{0}\rightarrow X$ be a trajectory of $\Sigma$. Suppose that for any $t\in\mathbb{R}_{+}^{0}$, $Tip(t,\mathbf{x})$
is a finite set and $\sigma$ and $\sigma_{t}$ are words corresponding to $\mathbf{x}$ and $\mathbf{x}^{t}$ (see Definition~\ref{Def:infinite word}), respectively.
Then the following conclusions hold:

(1) For any $j\in\mathbb{N}$ with $\sigma[j]\neq \sigma[j+1]$, there exist $t_{0},t_{0}'\in\mathbb{R}_{+}^{0}$ with $t_{0}<t_{0}'$ such that one of the following holds:

\ \ \  \ (a) $\sigma[j+1,\infty)=\sigma_{t_{0}}$ and for any $t'<t_{0}$, $\sigma_{t'}=\sigma[i,\infty)$ for some $i\leq j$,

\ \ \  \ (b) $\sigma[j+1,\infty)=\sigma_{t'}$ for all $t'\in(t_{0},t_{0}']$ and for any $t''\leq t_{0}$, $\sigma_{t''}=\sigma[i,\infty)$ for some $i\leq j$.

(2) For any $t\in\mathbb{R}_{+}^{0}$, there exists $j\in\mathbb{N}$ such that $\sigma[j,\infty)=\sigma_{t}$ and for any $i<j$, $\sigma[i,\infty)=\sigma_{t'}$ for some $t'<t$.
\end{lemma}
\begin{proof}
Since $\sigma$ is the word  corresponding to $\mathbf{x}$,  by Definition~\ref{Def:infinite word}, there exist $t_{i}\in\mathbb{R}_{+}^{0}$ $(i\in\mathbb{N})$ with $0=t_{1}<t_{2}<\cdots$ such that for any $i\in\mathbb{N}$, ($1_{i}$), ($2_{i}$) and ($3_{i}$) in Definition~\ref{Def:infinite word} hold for $\sigma$ and $\mathbf{x}$. In the following, we prove (1) and (2) in turn.

(1) Let $j\in\mathbb{N}$ and $\sigma[j]\neq \sigma[j+1]$.
Then by Definition~\ref{Def:infinite word}, there exists $t\in[t_{j},t_{j+1}]$ such that one of the following holds:

($i$) $\sigma[j]=\prod_h ({\mathbf{x}(t')})$ and $\sigma[j+1]=\prod_h({\mathbf{x}(t'')})$ for all $t'\in [t_{j},t)$ and ${t''\in [t,t_{j+1}]}$;

($ii$) $\sigma[j]=\prod_h({\mathbf{x}(t')})$ and $\sigma[j+1]=\prod_h({\mathbf{x}(t'')})$ for all $t'\in [t_{j},t]$ and ${t''\in (t,t_{j+1}]}$.

Suppose that $(i)$ holds.
We will show that $\sigma[j+1,\infty)=\sigma_{t}$ and for any $t'<t$, $\sigma_{t'}=\sigma[i,\infty)$ for some $i\leq j$.

To prove $\sigma[j+1,\infty)=\sigma_{t}$, we set $t_{1}'=0$ and for all $k\in\mathbb{N}$ with $k>1$, we set $t_{k}'=t_{j+k}-t$.
Further, we set $\sigma'=\prod_h(\mathbf{x}^{t}(t_{1}'))\prod_h(\mathbf{x}^{t}(t_{2}'))\prod_h(\mathbf{x}^{t}(t_{3}'))\cdots$.
Then it follows from ($i$) that
$\sigma'=\sigma[j+1,\infty)$.
Moreover, by Definition~\ref{Def:infinite word}, it is easy to check that $\sigma'$ is a word corresponding to $\mathbf{x}^{t}$.
Thus by (1) in Proposition~\ref{exist}, we obtain $\sigma'=\sigma[j+1,\infty)=\sigma_{t}$.

In the following, we demonstrate that for any $t'<t$, $\sigma_{t'}=\sigma[i,\infty)$ for some $i\leq j$. Let $t'<t$. Clearly, $t'\in[t_{n},t_{n+1})$ for some $n\leq j$. If $\prod_h(\mathbf{x}(t'))=\prod_h(\mathbf{x}(t_{n}))$ we set $i=n$, otherwise we set $i=n+1$.
Then by $(i)$ and $t'<t$, we get $i\leq j$.
Similar to the above,
we set $t_{1}''=0$ and $t_{k}''=t_{i+k-1}-t'$ for any $k\in\mathbb{N}$.
Then we set
$\sigma''=\prod_h(\mathbf{x}^{t'}(t_{1}''))\prod_h(\mathbf{x}^{t'}(t_{2}''))
\prod_h(\mathbf{x}^{t'}(t_{3}''))\cdots$.
Similar to the above, we may illustrate $\sigma''=\sigma[i,\infty)=\sigma_{t'}$.

Similarly, if ($ii$) holds, we may show that $\sigma[j+1,\infty)=\sigma_{t'}$ for all $t'\in(t,t_{j+1}]$ and for any $t''\leq t$, $\sigma_{t''}=\sigma[i,\infty)$ for some $i\leq j$.

(2) Let $t\in\mathbb{R}_{+}^{0}$. Consider the following two cases.

\textbf{Case 1}. $t\in[t_{i},t_{i+1})$ for some $i\in\mathbb{N}$.
Similar to (1), we may have $\sigma_{t}=\sigma[j,\infty)$ for $j=i$ or $j=i+1$.
Let $k<j$. We set $t'=t_{k}$. Similar to (1), we may get $\sigma_{t'}=\sigma[k,\infty)$.

\textbf{Case 2}. $t\not\in[t_{i},t_{i+1})$ for any $i\in\mathbb{N}$. Then it follows that $t>t_{i}$ for all $i\in\mathbb{N}$.
By Definition~\ref{Def:infinite word} and~\ref{tipping}, for any $i\in\mathbb{N}$, if $\prod_h(\mathbf{x}(t_{i}))\neq \prod_h(\mathbf{x}(t_{i+1}))$ then there exists at least one tipping point $t'\in[t_{i},t_{i+1}]$.
Further, since $Tip(t,\mathbf{x})$ is finite, there exists $j\in\mathbb{N}$ such that $\prod_h(\mathbf{x}(t_{j}))= \prod_h(\mathbf{x}(t_{j+1}))$.
Thus by Definition~\ref{Def:infinite word}, we have $\prod_h(\mathbf{x}(t''))=\sigma[j]=\sigma[i]$ for all $t''\geq t_{j}$ and $i\geq j$.
Then it follows from $t>t_{j}$ that $\prod_h(\mathbf{x}(t''))=\prod_h(\mathbf{x}(t))$ for all $t''\geq t$.
So by Definition~\ref{Def:infinite word}, it is easy to see that $\sigma_{t}=\sigma[j,\infty)$.

Let $i<j$. Clearly, $t_{i}<t$. Similar to (1), we may show that $\sigma_{t_{i}}=\sigma[i,\infty)$.
\qquad
\end{proof}

The following result demonstrates that, given a trajectory $\mathbf{x}$, Definition~\ref{Def:satisfaction of x} coincides with Definition~\ref{Def:satisfaction of x'} under the assumption that $Tip(t,\mathbf{x})$
is a finite set for any $t\in \mathbb{R}_{+}^{0}$.

\begin{proposition}
Let $\Sigma$ be a linear control system with state space $X$ and let $\mathbf{x}:\mathbb{R}_{+}^{0}\rightarrow X$ be a trajectory of $\Sigma$. If $Tip(t,\mathbf{x})$
is a finite set for any $t\in \mathbb{R}_{+}^{0}$ then for any LTL$_{-X}$ formula $\varphi$, $\mathbf{x}\models\varphi$ if and only if the word corresponding to $\mathbf{x}$ satisfies $\varphi$.
\end{proposition}
\begin{proof}
Suppose that $Tip(t,\mathbf{x})$
is a finite set for any $t\in \mathbb{R}_{+}^{0}$ and $\sigma$ is the word corresponding to $\mathbf{x}$.
It is enough to show that for any LTL$_{-X}$ formula $\varphi$ and $t\in\mathbb{R}_{+}^{0}$,
$\mathbf{x}^{t}\models\varphi$ if and only if $\sigma_{t}\models\varphi$,
where $\sigma_{t}$ is the word corresponding to $\mathbf{x}^{t}$.
We will proceed by induction on the structure of formula $\varphi$.
The proof is a routine case analysis. We will give two sample cases.

 \textbf{Case 1}. $\varphi=p$.
Let $t\in \mathbb{R}_{+}^{0}$. Then we have
\begin{equation*}
\begin{aligned}
\mathbf{x}^{t}\models\varphi &\textrm{ iff } \mathbf{x}(t)\models p
\\
& \textrm{ iff } \mathbf{x}(t)\in p &\textrm{ (by Definition~\ref{Def:satisfaction of x'})}
\\
& \textrm{ iff } p\in \sigma_{t}[1] &\textrm{ (by Definition~\ref{Def:infinite word})\ \ }
\\
& \textrm{ iff } \sigma_{t}\models p. &\textrm{ (by Definition~\ref{Def:satisfaction})\ \ }
\end{aligned}
\end{equation*}

 \textbf{Case 2}. $\varphi=\varphi_{1}\mathbf{U}\varphi_{2}$.
Let $t\in\mathbb{R}_{+}^{0}$.
We prove that $\mathbf{x}^{t}\models\varphi$ if and only if $\sigma_{t}\models\varphi$ as follows.

(From Left to Right)
Let $\mathbf{x}^{t}\models\varphi$. So $\mathbf{x}({t})\models\varphi$. Then by Definition~\ref{Def:satisfaction of x'},  there exist $t_{1},t_{2}\in \mathbb{R}_{+}^{0}$ with $t\leq t_{1}< t_{2}$ such that one of the following holds:

(a) $\mathbf{x}(t_{1})\models\varphi_{2}$ and $\mathbf{x}(t')\models\varphi_{1}$ for all $t'\in[t,t_{1})$,

(b) $\mathbf{x}(t')\models\varphi_{1}$ and $\mathbf{x}(t'')\models\varphi_{2}$ for all $t'\in[t,t_{1}]$ and  $t''\in(t_{1},t_{2}]$.

Suppose that (a) holds.
Then it follows that $\mathbf{x}^{t_{1}}\models \varphi_{2}$ and $\mathbf{x}^{t'}\models \varphi_{1}$ for any $t'\in[t,t_{1})$.
So by induction hypothesis, we obtain
\begin{equation}\label{eq:pro}
\sigma_{t_{1}}\models \varphi_{2}\textrm{ and }\sigma_{t'}\models \varphi_{1}\textrm{ for any }t'\in[t,t_{1}).
\end{equation}
Then by (2) in Lemma~\ref{Lem:subtrajectory and sub word}, there exists $j\in\mathbb{N}$ such that $\sigma_{t}[j,\infty)=\sigma_{t_{1}}$ and for any $i<j$, $\sigma_{t}[i,\infty)=\sigma_{t'}$ for some $t'\in[t,t_{1})$.
Further, it follows from (\ref{eq:pro}) and Definition~\ref{Def:satisfaction} that $\sigma_{t}[j]\models \varphi_{2}$ and $\sigma_{t}[i]\models \varphi_{1}$ for any $i<j$.
Therefore, by Definition~\ref{Def:satisfaction}, we get  $\sigma_{t}[1]\models \varphi$ and then $\sigma_{t}\models \varphi$.

Suppose that (b) holds. Then we have $\mathbf{x}^{t'}\models\varphi_{1}$ and $\mathbf{x}^{t''}\models\varphi_{2}$ for all $t'\in[t,t_{1}]$ and  $t''\in(t_{1},t_{2}]$.
So it follows from induction hypothesis that
\begin{equation}\label{eq:satis}
\sigma_{t'}\models \varphi_{1} \textrm{ and } \sigma_{t''}\models \varphi_{2} \textrm{ for all }t'\in[t,t_{1}]\textrm{ and  }t''\in(t_{1},t_{2}].
\end{equation}
Moreover, by (2) in Lemma~\ref{Lem:subtrajectory and sub word}, there exists $j\in\mathbb{N}$ such that $\sigma_{t}[j,\infty)=\sigma_{t_{2}}$ and for any $i<j$, $\sigma_{t}[i,\infty)=\sigma_{t_{i}'}$ for some $t_{i}'\in  \mathbb{R}_{+}^{0}$ with $t\leq t_{i}'<t_{2}$.
If $\sigma_{t}[i,\infty]\models \varphi_{1}$ for all $i<j$ then $\sigma_{t}\models \varphi$ holds trivially.
Suppose that $\sigma_{t}[n,\infty]\not\models \varphi_{1}$ for some $n<j$. Clearly, there exists $k\leq n$ such that $\sigma_{t}[k]\not\models \varphi_{1}$ and $\sigma_{t}[i]\models \varphi_{1}$ for all $i<k$.
Then since $k\leq n<j$, there exists $t_{k}'\in[t,t_{2})$ such that $\sigma[k,\infty)=\sigma_{t_{k}'}$.
Thus it follows from~(\ref{eq:satis}) and $\sigma_{t}[k]\not\models \varphi_{1}$ that $\sigma_{t}[k]\models \varphi_{2}$.
Therefore, since  $\sigma_{t}[i]\models \varphi_{1}$ for all $i<k$, we obtain $\sigma_{t}\models \varphi$.

(From Right to Left)
Let $\sigma_{t}\models \varphi_{1}\mathbf{U}\varphi_{2}$.
Then by Definition~\ref{Def:satisfaction}, there exists $n\in\mathbb{N}$ such that
$\sigma_{t}[n]\models \varphi_{2}$ and $\sigma_{t}[i]\models \varphi_{1}$ for any $i<n$.
Thus there exists $j\leq n$ such that
\begin{equation}\label{eq:pro3}
\sigma_{t}[j,\infty]\models \varphi_{2}, \sigma_{t}[i,\infty]\not\models \varphi_{2}\textrm{ and }\sigma_{t}[i,\infty]\models \varphi_{1}\textrm{ for any }i<j.
\end{equation}

If $j=1$ then $\sigma_{t}\models \varphi_{2}$. Further, by induction hypothesis, we obtain $\mathbf{x}^{t}\models\varphi_{2}$.
Then it follows from Definition~\ref{Def:satisfaction of x'} that
$\mathbf{x}^{t}\models\varphi_{1}\mathbf{U}\varphi_{2}$.
In the following, we consider the case where $j>1$.
Then by~(\ref{eq:pro3}) and Definition~\ref{Def:satisfaction}, it is easy to check that
$\sigma_{t}[j]\neq\sigma_{t}[j-1]$.
Thus by (1) in Lemma~\ref{Lem:subtrajectory and sub word},  there exists $t_{0},t_{0}'\in\mathbb{R}_{+}^{0}$ with $t_{0}<t_{0}'$ such that one of the following holds:

\ \ \ ($a$) $\sigma_{t}[j,\infty)=\sigma_{t+t_{0}}$ and for any $t'<t_{0}$, $\sigma_{t}[i,\infty)=\sigma_{t+t'}$ for some $i< j$,

\ \ \ ($b$) $\sigma_{t}[j,\infty)=\sigma_{t+t'}$ for all $t'\in(t_{0},t_{0}']$ and for any $t''\leq t_{0}$, $\sigma_{t}[i,\infty)=\sigma_{t+t''}$ for some $i< j$.

If ($a$) holds then it follows from induction hypothesis and (\ref{eq:pro3}) that $\sigma_{t+t_{0}}\models\varphi_{2}$ and $\sigma_{t+t'}\models\varphi_{1}$ for any $t'<t_{0}$.
Thus by Definition~\ref{Def:satisfaction of x'}, we obtain $\sigma_{t}\models\varphi$.
Similarly, if ($b$) holds, we may show that  $\sigma_{t}\models\varphi$.
\qquad
\end{proof}

Henceforth, the sentence ``trajectory $\mathbf{x}$ satisfies an LTL$_{-X}$ formula $\varphi$'' means $\mathbf{x}\models\varphi$ defined in Definition~\ref{Def:satisfaction of x'}.
\section{Transforming Specification}\label{Sec:logical}
The remainder of this paper concerns itself with the relationship between the formal design of Pola and Tabuada's abstractions and that of linear systems with disturbance inputs.
Similar problem has been considered for systems without disturbances \cite{fain:2,kloe:1,tab:1,tab:2}.
Amongst, Tabuada and Pappas demonstrate the following two conclusions~\cite{tab:1}:

\textit{(TP-1). There exists a controller for linear system enforcing specification if and only if there exists a controller for finite abstraction enforcing the same specification.}

\textit{(TP-2). The controller for finite abstraction can be applied to the original linear system to meet specification.}

Based on these two conclusions, in order to obtain a controller of control system enforcing the given specification, it is enough to construct a controller for finite abstraction enforcing this specification~\cite{tab:1}.
Unfortunately, when we consider linear system with disturbances, neither (TP-1) nor (TP-2) always holds. Two counterexamples are provided below.
\begin{figure}[t]
\begin{center}
\centerline{\includegraphics[scale=0.5]{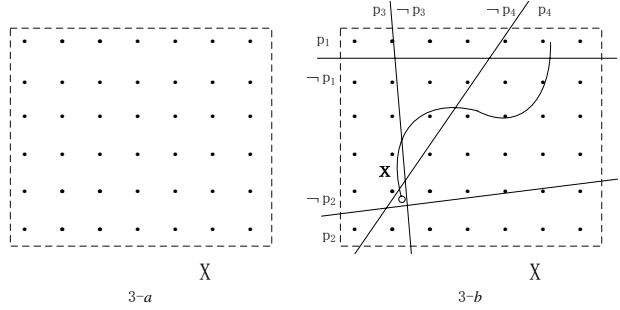}}\vspace{-0.7in}
\end{center}
\caption{Counterexample for (TP-1)}\label{Example: counter}\vspace{-0.3in}
\end{figure}
\begin{example}\label{example: counter1}
Consider the state space $X$ of linear system $\Sigma$, as shown in Fig~\ref{Example: counter}-a.
Given $\varepsilon\in\mathbb{R}_{+}$, let $\tau,\eta,\mu\in\mathbb{R}_{+}$ such that $\|e^{\mathbf{A}\tau}\|\varepsilon +\mu+\eta/2 < \varepsilon$.
Clearly, such $\tau,\eta,\mu$ exist.
Then by Theorem~\ref{Th:bisi}, any finite abstraction $T\in T_{\tau,\eta,\mu}(\Sigma)$ is $A\varepsilon A$ bisimilar to $T_{\tau}(\Sigma)$.
Let $T\in T_{\tau,\eta,\mu}(\Sigma)$ and $T=(Q,A,B,\rightarrow,\mathbb{R}^{n},H)$.
In Fig~\ref{Example: counter}-a, black spots denote the states of $T$.
Let $\mathbb{P}\!_h=\{p_{1},p_{2},p_{3},p_{4}\}$ be a finite set of propositions and let $p_{i}$ ($i=1,2,3,4$) be atomic proposition representing open half-space as illustrated in Fig~\ref{Example: counter}-b.
In this case, if specification $\varphi_{0}$ is $\neg p_{1}\wedge\neg p_{2}\wedge p_{3} \wedge p_{4}$, then
there exist some initial states of $\Sigma$ such that the trajectories of $\Sigma$ from these states satisfy specification (e.g., see $\mathbf{x}$ in Fig~\ref{Example: counter}-b).
Thus we may construct a controller which sets initial state of $\Sigma$ to be $\mathbf{x}(0)$.
Clearly, the trajectories of $\Sigma$ with this controller satisfy the above specification.
On the other hand, since every state in $Q$ (i.e., black spots in Fig~\ref{Example: counter}-a) doesn't satisfy $\varphi_0$, any trajectory of $T$ does not satisfy this specification.
So there does not exist a controller for $T$ enforcing this specification.
Therefore, (TP-1) does not always hold for Pola and Tabuada's abstractions and linear systems with disturbance inputs.
\end{example}
\begin{figure}[t]
\begin{center}
\centerline{\includegraphics[scale=0.5]{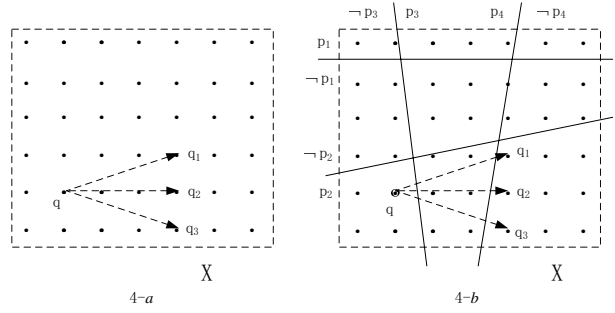}}\vspace{-0.7in}
\end{center}
\caption{Counterexample for (TP-2)}\label{Example: counter'}\vspace{-0.3in}
\end{figure}
\begin{example}\label{example:counter2}
Similar to Example~\ref{example: counter1},  in Fig~\ref{Example: counter'}-a, $X$ denotes the state space of a linear system $\Sigma$.
Given $\varepsilon\in \mathbb{R}_{+}$, let $\tau,\eta,\mu\in\mathbb{R}_{+}$ with $\|e^{\mathbf{A}\tau}\|\varepsilon +\mu+\eta/2 < \varepsilon$.
 Clearly, such $\tau,\eta,\mu$ exist.
Thus any finite abstraction
 $T\in T_{\tau,\eta,\mu}(\Sigma)$ is $A\varepsilon A$ bisimilar to $T_{\tau}(\Sigma)$.
Let  $T\in T_{\tau,\eta,\mu}(\Sigma)$ and $T=(Q,A,B,\rightarrow,\mathbb{R}^{n},H)$.
The states of finite abstraction are indicated by  black spots in {Fig~\ref{Example: counter'}-a}.
Let $q\in Q$ be a state of $T$.
Without loss generality, we may suppose that $a\in A$ is a control label of $T$ and $\{{q':q\xrightarrow{a,b}q'} \ \textrm{for some disturbance} \textrm{ label }  b\in B\}=\{q_{1},q_{2},q_{3}\}$, as illustrated in {Fig~\ref{Example: counter'}-a}.
Consider a finite set $\mathbb{P}\!_h=\{p_{1},p_{2},p_{3},p_{4}\}$, where $p_{i}$ ($i=1,2,3,4$) is atomic proposition representing open half-space as shown in Fig~\ref{Example: counter'}-b.
Let the specification $\varphi_{0}=(\neg p_{3} \wedge p_{4})\mathbf{U} (p_{3} \wedge\neg p_{4})$.
We set the initial state to be $q$ and put the control label to be $a$ when the current state of $T$ is $q$.
Under such control, it is easy to check that the trajectories of $T$ satisfy the given specification.
However, due to Fig~\ref{Example: counter'}-b, it is clear that  any trajectory of $\Sigma$ does not satisfy this specification under any control.
Thus (TP-2) does not always hold for Pola and Tabuada's abstractions and linear systems with disturbance inputs.
\end{example}

Due to the above two examples, we know that  linear systems  and their finite abstractions do not always share the identical properties described by LTL$_{-X}$ formulae under control.
Thus, given an LTL$_{-X}$ specification $\varphi_0$ for linear systems with disturbance inputs, if we directly adopt $\varphi_0$ as specification for finite abstraction, then the formal design for the latter may not be helpful for the former.
The remainder of this paper will try to find a way to solve this problem and establish results similar to (TP-1) and (TP-2) for systems with disturbances.
To this end, we will transform LTL$_{-X}$ specification $\varphi_{0}$ for linear system $\Sigma$ to specification $\varphi_{0}''$ for finite abstraction and demonstrate that, under some assumptions, given an initial state $q_0$ and a control strategy $f$ of finite abstraction enforcing $\varphi_{0}''$, there exists a controller based on $q_0$ and $f$ so that  the trajectories of $\Sigma$ with this controller satisfy~$\varphi_{0}$.
This section will take two steps to realize such transformation.
\subsection{Transforming specifications for $\Sigma$ to specifications for $T_{\tau}(\Sigma)$}\label{Sec:logical connection 1}
This subsection will deal with transforming the specification $\varphi_{0}$ for $\Sigma$ to $\varphi_{0}'$ for $T_{\tau}(\Sigma)$. We will show that under some circumstance, if $\sigma_{{\mathbf{x}}}$ is a sampling trajectory of $\mathbf{x}$ then $\sigma_{{\mathbf{x}}} \models \varphi_{0}'$ implies $\mathbf{x}\models \varphi_{0}$.
Here the specification $\varphi_{0}'$ is described by the linear temporal logic defined below.
\begin{definition}\label{Def:LTL1 formula}
Let $\delta\in\mathbb{R}_{+}$.  The formulae $\varphi$ of linear temporal logic LTL$_{-X}^{\delta}(\mathbb{P}\!_h)$ are inductively defined as:
\begin{center}
$\varphi::=[\delta]p|[\delta]\neg p|\varphi_{1}\wedge\varphi_{2}|\varphi_{1}\vee\varphi_{2}|\varphi_{1}\mathbf{U}\varphi_{2}|\varphi_{1}\tilde{\mathbf{U}}\varphi_{2}$
\end{center}
where $p\in \mathbb{P}\!_h$, i.e., $p=\{x\in\mathbb{R}^{n}:c_{p}^{T}x+d_{p}<0\}$ for some $c_{p}\in\mathbb{R}^{n}$ and $d_{p}\in\mathbb{R}$.
\end{definition}

The semantics of LTL$_{-X}^{\delta}$ formulas are defined as follows.

\begin{definition}\label{Def:satisfaction2}
Let $\sigma\in (\mathbb{R}^{n})^{\omega}$ and $\delta\in\mathbb{R}_{+}$.
The satisfaction of LTL$_{-X}^{\delta}$ formula $\varphi$ at position $i\in \mathbb{N}$ of $\sigma$, denoted by $\sigma[i]\models\varphi$, is defined similarly to Definition~\ref{Def:satisfaction} except for the cases where either $\varphi=[\delta]p$ or $\varphi=[\delta]\neg p$:

($1'$) $\sigma[i]\models [\delta]p$ iff ${\cal{B}}_{\delta}(\sigma[i])\subseteq p$;

($2'$) $\sigma[i]\models [\delta]\neg p$ iff ${\cal{B}}_{\delta}(\sigma[i])\cap p=\emptyset$.

The infinite sequence $\sigma$ satisfies an LTL$_{-X}^{\delta}$ formula $\varphi$, written as $\sigma\models\varphi$, if and only if $\sigma[1]\models \varphi$.
\end{definition}

In order to transform $\varphi_{0}$ to the desired $\varphi_{0}'$, we introduce the following function.

\begin{definition}\label{Def:transformation}
Let $\delta\in\mathbb{R}_{+}$. The function $tr_{\delta}:\mathrm{LTL}_{-X}\rightarrow\mathrm{LTL}_{-X}^{\delta}$ is inductively defined as follows:

(1) $tr_{\delta}(p)=[\delta]p$;

(2) $tr_{\delta}(\neg p)=[\delta]\neg p$;

(3) $tr_{\delta}(\varphi_{1}\wedge\varphi_{2})=tr_{\delta}(\varphi_{1})\wedge tr_{\delta}(\varphi_{2})$;

(4) $tr_{\delta}(\varphi_{1}\vee\varphi_{2})=tr_{\delta}(\varphi_{1})\vee tr_{\delta}(\varphi_{2})$;

(5) $tr_{\delta}(\varphi_{1}\mathbf{U}\varphi_{2})=tr_{\delta}(\varphi_{1})\mathbf{U} tr_{\delta}(\varphi_{2})$;

(6) $tr_{\delta}(\varphi_{1}\tilde{\mathbf{U}}\varphi_{2})=tr_{\delta}(\varphi_{1})\tilde{\mathbf{U}} tr_{\delta}(\varphi_{2})$.
\end{definition}

The following result reveals that, for any LTL$_{-X}$ formula $\varphi_{0}$, under some assumption, if the sample trajectory satisfies $tr_{\delta}(\varphi_{0})$ then the original trajectory of $\Sigma$ satisfies $\varphi_{0}$.

\begin{theorem}\label{Th:logical connection 1'}
Let $\Sigma$ be a linear control system with state space $X$, $\mathbf{x}:\mathbb{R}_{+}^{0}\rightarrow X$ a trajectory of $\Sigma$, $\sigma_{{\mathbf{x}}} = \mathbf{x}(0)\mathbf{x}(\tau)\mathbf{x}(2\tau)\cdots$ and $\delta\in\mathbb{R}_{+}$.
If $\|\mathbf{x}(t)-\mathbf{x}((n-1)\tau)\|\leq \delta$ for any $n\in\mathbb{N}$ and $t\in [(n-1)\tau,n\tau)$, then for any LTL$_{-X}$ formula $\varphi_{0}$, ${\sigma_{{\mathbf{x}}}\models tr_{\delta}(\varphi_{0})}$ implies $\mathbf{x}\models\varphi_{0}$.
\end{theorem}
\begin{proof}
Suppose that $\|\mathbf{x}(t)-\mathbf{x}((n-1)\tau)\|\leq \delta$ for any $n\in\mathbb{N}$ and $t\in [(n-1)\tau,n\tau)$.
To complete the proof, it is enough to show that  for any LTL$_{-X}$ formula $\varphi_{0}$ and for any $i\in\mathbb{N}$ and $t\in\mathbb{R}_{+}^{0}$ with $(i-1)\tau\leq t\leq i\tau$, if ${\sigma_{{\mathbf{x}}}[i]\models tr_{\delta}(\varphi_{0})}$ then $\mathbf{x}^{t}\models\varphi_{0}$, where $\mathbf{x}^{t}(t')=\mathbf{x}(t+t')$ for any $t'\in\mathbb{R}_{+}^{0}$.
We proceed by induction on $\varphi_{0}$. The proof is a routine case analysis. We give two sample cases.

 \textbf{Case 1}. $\varphi_{0}=p$.
Let $i\in\mathbb{N}$, $(i-1)\tau\leq t\leq i\tau$ and ${\sigma_{{\mathbf{x}}}[i]\models tr_{\delta}(\varphi_{0})}$.
Then by Definition~\ref{Def:transformation}, we obtain ${\sigma_{{\mathbf{x}}}[i]\models [\delta]p}$.
It follows from Definition~\ref{Def:satisfaction2} that ${\cal{B}}_{\delta}(\sigma_{{\mathbf{x}}}[i])\subseteq p$.
Then since $\|\mathbf{x}(t')-\mathbf{x}((n-1)\tau)\|\leq \delta$ for any $n\in\mathbb{N}$ and $t'\in [(n-1)\tau,n\tau)$, we have $\mathbf{d}(\sigma_{{\mathbf{x}}}[i], \mathbf{x}(t))\leq\delta$.
This, together with ${\cal{B}}_{\delta}(\sigma_{{\mathbf{x}}}[i])\subseteq p$, implies that $\mathbf{x}(t)\in p$.
Thus by Definition~\ref{Def:satisfaction of x'}, we get $\mathbf{x}(t)\models\varphi_{0}$ and then $\mathbf{x}^{t}\models\varphi_{0}$.

 \textbf{Case 2}. $\varphi_{0}=\varphi_{1}\mathbf{U}\varphi_{2}$.
Let $i\in\mathbb{N}$, $(i-1)\tau\leq t\leq i\tau$ and ${\sigma_{{\mathbf{x}}}[i]\models tr_{\delta}(\varphi_{0})}$.
Then it follows from Definition~\ref{Def:transformation} that ${\sigma_{{\mathbf{x}}}[i]\models tr_{\delta}(\varphi)_{1}\mathbf{U}tr_{\delta}(\varphi)_{2}}$.
Thus by Definition~\ref{Def:satisfaction2}, there exists $j\geq i$ such that $\sigma_{{\mathbf{x}}}[j]\models\varphi_{2}$ and for all $k\in\mathbb{N}$ with $i\leq k<j$, we have  $\sigma_{{\mathbf{x}}}[k]\models\varphi_{1}$.
So by induction hypothesis, we obtain $\mathbf{x}^{(j-1)\tau}\models\varphi_{2}$ and $\mathbf{x}^{t_{1}}\models \varphi_{1}$ for any $k\in\mathbb{N}$ and $t_{1}\in\mathbb{R}_{+}^{0}$ with $i\leq k<j$ and $(k-1)\tau \leq t_{1}< k\tau$.
Then it follows that $\mathbf{x}((j-1)\tau)\models\varphi_{2}$ and $\mathbf{x}(t_{1})\models \varphi_{1}$ for any $k\in\mathbb{N}$ and $t_{1}\in\mathbb{R}_{+}^{0}$ with $i\leq k<j$ and $(k-1)\tau \leq t_{1}< k\tau$.
Therefore, by Definition~\ref{Def:satisfaction of x'}, we get $\mathbf{x}(t)\models \varphi_{1}\mathbf{U}\varphi_{2}$ and then $\mathbf{x}^{t}\models \varphi_{1}\mathbf{U}\varphi_{2}$.
\qquad
\end{proof}
\subsection{Transforming specifications for $T_{\tau}(\Sigma)$ to ones for $T_{\tau,\eta,\mu}(\Sigma)$}\label{Sec:logical connection 2}
This subsection will concern itself with the transformation from $tr_{\delta}(\varphi_{0})$ for $T_{\tau}(\Sigma)$ to specification $\varphi_{0}''$ for finite abstractions of $\Sigma$.
Similar to the function $tr_{\delta}$, we introduce a transform function below.

\begin{definition}\label{Def:transformation 2}
Let $\varepsilon, \delta \in \mathbb{R}_{+}$. The function $tr_{\varepsilon}^{\delta}:\mathrm{LTL}_{-X}^{\delta}\rightarrow\mathrm{LTL}_{-X}^{\delta +\varepsilon}$  is defined as for each LTL$_{-X}^{\delta}$ formula $\psi$, $tr_{\varepsilon}^{\delta}(\psi)$ is obtained from $\psi$ by replacing $[\delta]$ by $[\delta +\varepsilon]$.
\end{definition}

In the rest of this subsection, we want to show that under some assumptions, for any $\varepsilon, \delta,\tau,\eta,\mu \in \mathbb{R}_{+}$, finite abstraction $T\in T_{\tau,\eta,\mu}(\Sigma)$ and LTL$_{-X}^{\delta}$ formula $\psi$, if specification $tr_{\varepsilon}^{\delta} (\psi)$ is satisfied by $T$ under control, then $\psi$ is satisfied by $T_{\tau}(\Sigma)$ under control.
To this end, some notions related to control strategy are introduced below.

\begin{definition}\label{Def:strategy}
A control strategy for an alternating transition system $T=(Q,A,B,\longrightarrow,O,H)$ is a function $f:Q^{+}\rightarrow 2^{A}-\{\emptyset\}$. For any $q\in Q$, the outcomes $Out^{n}_{T}(q,f)$ $(n\in\mathbb{N})$ and $Out_{T}(q,f)$  of $f$ from $q$ are defined as follows:
\begin{equation*}
\begin{aligned}Out^{n}_{T}(q,f)& =\{s\in Q^{n}:  s[1]=q\ \mathrm{and} \forall 1\leq i<n \exists a_{i}\in f(s[1,i]) \exists b_{i}\in B (s[i]\xrightarrow{a_{i},b_{i}} s[i+1])\},
 \\
 Out_{T}(q,f)&=\{\sigma\in Q^{\omega}: \sigma[1] =q\ \mathrm{and} \forall i\in\mathbb{N} \exists a_{i}\in f(\sigma [1,i]) \exists b_{i}\in B (\sigma [i]\xrightarrow{a_{i},b_{i}} \sigma[i+1])\}.
\end{aligned}
\end{equation*}
Furthermore, we define $Out^{+}_{T}(q,f)$ as:
$Out^{+}_{T}(q,f)=\bigcup_{n\in\mathbb{N}}Out^{n}_{T}(q,f)$.
\end{definition}

If alternating transition system $T$ is known from the context, we often omit subscripts in $Out^{n}_{T}(q,f)$, $Out_{T}(q,f)$ and $Out^{+}_{T}(q,f)$.

\begin{lemma}\label{Lem:outcomes}
Let $k\in \mathbb{N}$. Then $Out^{k+1}_{T}(q,f)=\{s\in Q^{k+1}:  s[1,k]\in Out^{k}_{T}(q,f)\textrm{ and }\exists a_{k}\in f(s[1,k]) \exists b_{k}\in B (s[k]\xrightarrow{a_{k},b_{k}} s[k+1])
\}$.\end{lemma}
\begin{proof}
Straightforward.
\end{proof}

\begin{definition}\label{Def:control satisfy}
Let $\Sigma$ be a linear control system and $q$ a state of $T_{\tau}(\Sigma)$. We say that the formula $\varphi$ is satisfied by $q$ under control if and only if there exists a control strategy $f$ such that $\sigma\models\varphi$ for all $\sigma\in Out(q,f)$.
Furthermore, we say that the formula $\varphi$ is satisfied by $T_{\tau}(\Sigma)$ under control if and only if there exists a state $q$ of $T_{\tau}(\Sigma)$ such that $\varphi$ is satisfied by $q$ under control.

Let  $\tau,\eta,\mu\in\mathbb{R}_{+}$, $T\in T_{\tau,\eta,\mu}(\Sigma)$ and let $q'$ be a state of $T$. Similarly, we may define that the formula $\varphi$ is satisfied by  $q'$ and $T$ under control.
\end{definition}

\begin{lemma}\label{Lem:bisimilar state to bisimilar sequence}
Let $T_{i}=(Q_{i},A_{i}\times B_{i}, \xrightarrow[i]{}, O, H_{i})$ ($i=1,2$) be two metric, non-blocking alternating transition systems with the same observation set and the same metric $\textbf{d}$ over $O$.
Suppose that $Q_1$ is finite and $f:(Q_{1})^{+}\rightarrow 2^{A_{1}}-\{\emptyset\}$ is a control strategy.
For any $q_{1}\in Q_{1}$, $q_{2}\in Q_{2}$ and $\varepsilon\in\mathbb{R}_{+}$, if $q_{1}\sim_{\varepsilon} q_{2}$ then there exists a control strategy $f':(Q_{2})^{+}\rightarrow 2^{A_{2}}-\{\emptyset\}$ such that for any $\sigma_{2}\in Out(q_{2},f')$, $\sigma_{1}\sim_{\varepsilon}\sigma_{2}$ for some $\sigma_{1}\in Out(q_{1},f)$ \footnote{For any (finite or infinite) sequences $\alpha_{1}$ and $\alpha_{2}$, $\alpha_{1}\sim_{\varepsilon}\alpha_{2}$ if and only if $|\alpha_{1}|=|\alpha_{2}|$ and $\alpha_{1}[i]\sim_{\varepsilon}\alpha_{2}[i]$ for  all $i\leq |\alpha_{1}|$.}.
\end{lemma}
\begin{proof}
Let $\varepsilon\in \mathbb{R}_{+}$, $q_{1}\in Q_{1}$, $q_{2}\in Q_{2}$ and $q_{1}\sim_{\varepsilon} q_{2}$. In order to obtain the desired control strategy $f'$, we define the subset $\triangle_{n}$ of $(Q_{2})^{n}$ and the function $f_{n}:\triangle_{n}\rightarrow 2^{A_{2}}$ $(n\in \mathbb{N})$ inductively as follows:

 $\triangle_{1}=\{q_{2}\}$ and the function $f_{1}:\triangle_{1}\rightarrow 2^{A_{2}}$ is defined as
\begin{equation*}
\begin{aligned} f_{1}(q_{2})\triangleq\{a_{2}\in A_{2}:  & \exists a_{1}\in f(q_{1})\forall b_{2}\in B_{2} \forall q_{2}'\in Q_{2} \ \
\ \\
& (q_{2}\xrightarrow[2]{a_{2},b_{2}}q_{2}'\Rightarrow \exists b_{1}\in B_{1}\exists q_{1}'(q_{1}\xrightarrow[1]{a_{1},b_{1}}q_{1}'\textrm{ and } q_{1}'\sim_{\varepsilon}q_{2}'))\}.
\end{aligned}
\end{equation*}

Assume that $\triangle_{k}$ and $f_{k}$ have been defined. Now we define $\triangle_{k+1}$ and  $f_{k+1}$ below:
\begin{center}
$\triangle_{k+1}\triangleq\{s_{2}q_{2}': s_{2}\in\triangle_{k}\
\mathrm{and}\ \exists a_{2}\in f_{k}(s_{2})\exists b_{2}\in B_{2} (s_{2}[k]\xrightarrow[2]{a_{2},b_{2}}q_{2}')\}$
\end{center}
and the function $f_{k+1}:\triangle_{k+1}\rightarrow 2^{A_{2}}$ is defined as for any $s_{2}'\in\triangle_{k+1}$,
\begin{equation*}
\begin{aligned} f_{k+1}(s_{2}')\triangleq\{a_{2}\in A_{2}:  & \exists s_{1}'\in Out^{k+1}(q_{1},f)(s_{1}'\sim_{\varepsilon} s_{2}' \ \mathrm{and}\ \exists a_{1}\in f(s_{1}')\forall b_{2}\in B_{2} \forall q_{2}'\in Q_{2} \\ & ( s_{2}'[end]\xrightarrow[2]{a_{2},b_{2}}q_{2}'\Rightarrow  \exists b_{1}\in B_{1} \exists q_{1}'(s_{1}'[end]\xrightarrow[1]{a_{1},b_{1}}q_{1}'\textrm{ and } q_{1}'\sim_{\varepsilon}q_{2}'))\}.
\end{aligned}
\end{equation*}
Based on the above definition, we may define $f':(Q_{2})^{+}\rightarrow 2^{A_{2}}$ as follows:
\begin{equation*}
f'(s)=\left\{\begin{aligned}& f_{|s|}(s)   &\textrm{  if } s\in \triangle_{|s|}\\
& A_{2}   &\textrm{otherwise}
\end{aligned}\right.
\end{equation*}
To show that $f'$ is the desired control strategy, we prove the following three claims in turn.

 \textbf{Claim 1}.
For any $n\in\mathbb{N}$, we have

($1_{n}$) $\triangle_{n}\neq\emptyset$;

($2_{n}$) for any $s_{2}\in\triangle_{n}$, there exists $s_{1}\in Out^{n}(q_{1},f)$ such that $s_{1}\sim_{\varepsilon} s_{2}$;

(3$_{n}$) for any $s_{2}\in\triangle_{n}$, $f_{n}(s_{2})\neq\emptyset$.

We proceed by induction on $n$.

If $n=1$ then $(1_{n})$ and $(2_{n})$ hold trivially. Since $f$ is a control strategy, we have $f(q_{1})\neq\emptyset$. Let $a_{1}\in f(q_{1})$. Then by $q_{1}\sim_{\varepsilon}q_{2}$ and Proposition~\ref{Pro:approximate bisimulation}, there exists $a_{2}\in A_{2}$ such that
 \begin{equation*}
 \forall b_{2}\in B_{2} \forall q_{2}'\in Q_{2}(q_{2}\xrightarrow[2]{a_{2},b_{2}}q_{2}'\Rightarrow \exists b_{1}\in B_{1}\exists q_{1}'(q_{1}\xrightarrow[1]{a_{1},b_{1}}q_{1}'\textrm{ and } q_{1}'\sim_{\varepsilon}q_{2}'))).
 \end{equation*}
Thus $a_{2}\in f_{n}(q_{2})$ and then $(3_{n})$ holds.

Suppose that $(1_{k})$, $2_{k}$ and $(3_{k})$ hold. We prove $1_{k+1}$, $2_{k+1}$ and $3_{k+1}$ in turn.

($1_{k+1}$) By induction hypothesis, we get $\triangle_{k}\neq\emptyset$ and $f_{k}(s_{2})\neq\emptyset$ for any $s_{2}\in\triangle_{k}$. Thus there exists $s_{2}\in\triangle_{k}$ and $a_{2}\in f_{k}(s_{2})$. Let $b_{2}\in B_{2}$. Then $s_{2}[k]\xrightarrow[2]{a_{2},b_{2}} q_{2}'$ for some $q_{2}'\in Q_{2}$. Therefore, $s_{2}q_{2}'\in\triangle_{k+1}$ and then $\triangle_{k+1}\neq\emptyset$.

($2_{k+1}$) Let $s_{2}'\in\triangle_{k+1}$. Then by the definition of $\triangle_{k+1}$, there exists $a_{2}\in f_{k}(s_{2}'[1,k])$ and $b_{2}\in B_{2}$ such that $s_{2}'[k]\xrightarrow[2]{a_{2},b_{2}}s_{2}'[k+1]$.
Further, by the definition of $f_{k}$, there exists $s_{1}\in Out^{k}(q_{1},f)$ such that $s_{1}\sim_{\varepsilon} s_{2}'[1,k]$, $s_{1}[k]\xrightarrow[1]{a_{1},b_{1}}q_{1}'$ and $q_{1}'\sim_{\varepsilon}s_{2}'[k+1]$ for some $a_{1}\in f(s_{1})$, $b_{1}\in B_{1}$ and $q_{1}'\in Q_{1}$.
Thus we obtain $s_{1}q_{1}'\in Out^{k+1}(q_{1},f)$ and $s_{1}q_{1}'\sim_{\varepsilon}s_{2}'$.

($3_{k+1}$) Let $s_{2}'\in\triangle_{k+1}$. By $(2_{k+1})$, there exists $s_{1}'\in Out^{k+1}(q_{1},f)$ such that $s_{1}'\sim_{\varepsilon} s_{2}'$. So we have $s_{1}'[end]\sim_{\varepsilon} s_{2}'[end]$. On the other hand, since $f$ is a control strategy, we get $f(s_{1}')\neq\emptyset$. Then similar to the case $n=1$, we may show that $f_{k+1}(s_{2}')\neq\emptyset$.

  \textbf{Claim 2}.
$f'$ is a control strategy and $\triangle_{n}=Out^{n}(q_{2},f')$ for any $n\in \mathbb{N}$.

It follows from Claim 1 and the definition of $f'$ that $f'(s_{2})\neq\emptyset$ for any $s_{2}\in Q_{2}^{+}$. Thus by Definition~\ref{Def:strategy}, $f'$ is a control strategy. Next, we show that $\triangle_{n}=Out^{n}(q_{2},f')$ for any $n\in \mathbb{N}$.

If $n=1$ then $\triangle_{1}=\{q_{2}\}=Out^{1}(q_{2},f')$. Let $n=k+1$. By induction hypothesis, we obtain $\triangle_{k}=Out^{k}(q_{2},f')$.
Moreover, it follows from the definition of $f'$ that $f'(s_{2})=f_{k}(s_{2})$ for all $s_{2}\in\triangle_{k}$.
Then, since $\triangle_{k}=Out^{k}(q_{2},f')$, by the definition of $\triangle_{k+1}$ and Lemma~\ref{Lem:outcomes}, it is clear that $\triangle_{k+1}=Out^{k+1}(q_{2},f')$.

  \textbf{Claim 3}.
For any $\sigma_{2}\in Out(q_{2},f')$, there is $\sigma_{1}\in Out(q_{1},f)$ such that $\sigma_{1}\sim_{\varepsilon}\sigma_{2}$.

Let  $\sigma_{2}\in Out(q_{2},f')$.
By Claim 2, for each $n\in\mathbb{N}$, we have $\sigma_{2}[1,n]\in\triangle_{n}$. Then, by Claim~1, there exist a family of sequences $s^{n}\in Out^{n}(q_{1},f)$ ($n\in\mathbb{N}$) such that $s^{n}\sim_{\varepsilon}\sigma_{2}[1,n]$ for each $n\in\mathbb{N}$.
Further, since $Q_1$ is finite, it is easy to check that there exists an infinite sequence $i_1 i_2\cdots \in \mathbb{N}^{\omega}$ such that for any $j\in \mathbb{N}$, $i_j<i_{j+1}$ and $s^{i_{j}}$ is a proper prefix of $s^{i_{j+1}}$, i.e., $s^{i_{j}}\circ s=s^{i_{j+1}}$ for some $s\in (Q_1)^{+}$.
Clearly, for any $k\in \mathbb{N}$, there exists $j\in\mathbb{N}$ such that $k<i_j$.
Furthermore, for any $j,l,k\in\mathbb{N}$, if $k<i_j$ and $k<i_l$, then $s_{i_j}[k]=s_{i_l}[k]$.
We define an infinite $\sigma_1\in (Q_{1})^{\omega}$ as: for any $k\in \mathbb{N}$, if $k<i_j$ for some $j\in\mathbb{N}$, then we set $\sigma_1[k]=s^{i_j}[k]$.
It is easy to see that $\sigma_1$ is well-defined.
Then, since $s^{i_j}\in Out^{+}(q_{1},f)$ and $s^{i_j}\sim_{\varepsilon}\sigma_{2}[1,i_j]$ for all $j\in\mathbb{N}$, by Definition~\ref{Def:strategy}, we have $\sigma_1\in Out(q_1,f)$ and $\sigma_1\sim_{\varepsilon}\sigma_{2}$.
\end{proof}

\begin{lemma}\label{Lem:logical characterization of bisimilar sequence}
Let $\Sigma$ be a linear control system, $\varepsilon,\delta, \tau, \eta,\mu\in\mathbb{R}_{+}$ and let $T\in T_{\tau, \eta,\mu}(\Sigma)$ be a finite abstraction of $\Sigma$. For any trajectory $\sigma_{1}$ of $T_{\tau}(\Sigma)$ and any trajectory $\sigma_{2}$ of $T$, if $\sigma_{1}\sim_{\varepsilon}\sigma_{2}$ then for any LTL$_{-X}^{\delta}$ formula $\psi$, $\sigma_{2}\models tr^{\delta}_{\varepsilon}(\psi)$ implies $\sigma_{1}\models \psi$.
\end{lemma}
\begin{proof}
We argue by induction on the structure of $\psi$. We give two sample cases.

 \textbf{Case 1}. $\psi=[\delta]p$.
Then by Definition~\ref{Def:transformation 2}, we have $tr^{\delta}_{\varepsilon}(\psi)=[\delta +\varepsilon]p$.
Let $\sigma_{1}\sim_{\varepsilon}\sigma_{2}$ and $\sigma_{2}\models tr^{\delta}_{\varepsilon}(\psi)$.
Therefore, $\mathbf{d}(\sigma_{1}[1],\sigma_{2}[1])|\leq\varepsilon$  and $\sigma_{2}\models [\delta +\varepsilon]p$. It follows from Definition~\ref{Def:satisfaction2} that $\sigma_{2}[1]\models[\delta +\varepsilon]p$.
To prove $\sigma_{1}\models \psi$, by Definition~\ref{Def:satisfaction2}, it is enough to show that $q\in p$ for any $q\in \mathbb{R}^{n}$ with $\mathbf{d}(q,\sigma_{1}[1])\leq\delta$.

Let $q\in \mathbb{R}^{n}$ and $\mathbf{d}(q,\sigma_{1}[1])\leq\delta$.
 Then it follows from $\mathbf{d}(\sigma_{1}[1],\sigma_{2}[1])|\leq\varepsilon$ that
 $\mathbf{d}(q,\sigma_{2}[1])\leq \mathbf{d}(q,\sigma_{1}[1])+\mathbf{d}(\sigma_{1}[1],\sigma_{2}[1])\leq\delta +\varepsilon$.
So by $\sigma_{2}[1]\models[\delta +\varepsilon]p$ and  Definition~\ref{Def:satisfaction2}, we get $q\in p$.

 \textbf{Case 2}. $\psi=\psi_{1}\mathbf{U}\psi_{2}$. It follows from Definition~\ref{Def:transformation 2} that $tr^{\delta}_{\varepsilon}(\psi)= tr^{\delta}_{\varepsilon}(\psi_{1})\mathbf{U}tr^{\delta}_{\varepsilon}(\psi_{2})$.
Let $\sigma_{1}\sim_{\varepsilon}\sigma_{2}$ and $\sigma_{2}\models tr^{\delta}_{\varepsilon}(\psi)$.
Thus by Definition~\ref{Def:satisfaction2}, for some $j\in\mathbb{N}$, we obtain
$\sigma_{2}[j]\models tr^{\delta}_{\varepsilon}(\psi_{2})$ and $\sigma_{2}[i]\models tr^{\delta}_{\varepsilon}(\psi_{1})$ for all $1\leq i<j$.
  Then by Definition~\ref{Def:satisfaction2}, $\sigma_{2}[j,\infty]\models tr^{\delta}_{\varepsilon}(\psi_{2})$ and $\sigma_{2}[i,\infty]\models tr^{\delta}_{\varepsilon}(\psi_{1})$ for $1\leq i<j$.
Moreover, it follows from $\sigma_{1}\sim_{\varepsilon}\sigma_{2}$ that $\sigma_{1}[j,\infty]\sim_{\varepsilon}\sigma_{2}[j,\infty]$ and $\sigma_{1}[i,\infty]\sim_{\varepsilon}\sigma_{2}[i,\infty]$ for all $1\leq i<j$.
Further, by induction hypothesis,
 $\sigma_{1}[j,\infty]\models \psi_{2}$ and $\sigma_{1}[i,\infty]\models \psi_{1}$ for all $1\leq i<j$.
 Thus it follows from Definition~\ref{Def:satisfaction2} that $\sigma_{1}[j]\models \psi_{2}$ and $\sigma_{1}[i]\models \psi_{1}$ for all ${1\leq i<j}$. Therefore, we have $\sigma_{1}[1]\models \psi_{1}\mathbf{U} \psi_{2}$ and then
$\sigma_{1}\models \psi$.  \qquad
\end{proof}

Now, we arrive at the main result of this subsection.

\begin{theorem}\label{Th:logical connection origin}
Given an asymptotically stable linear control system $\Sigma$ below
\begin{eqnarray*}
\Sigma : \dot{x}=\mathbf{A}x+\mathbf{B}u+\mathbf{G}v,\ \ \  x\in X,u\in U,v\in V
\end{eqnarray*}
and $\varepsilon,\delta\in\mathbb{R}_{+}$.
For any $\tau,\eta,\mu\in\mathbb{R}_{+}$ satisfying $\|e^{\mathbf{A}\tau}\|\varepsilon +\mu+\eta/2 < \varepsilon$
and for any $T\in T_{\tau, \eta,\mu}(\Sigma)$ and LTL$_{-X}^{\delta}$ formula $\psi$, if $tr_{\varepsilon}^{\delta}(\psi)$ is satisfied by $T$ under control then $\psi$ is satisfied by $T_{\tau}(\Sigma)$ under control.
\end{theorem}
\begin{proof}
Let $\tau,\eta,\mu \in\mathbb{R}_{+}$ such that $\|e^{\mathbf{A}\tau}\|\varepsilon +\mu+\eta/2 < \varepsilon$ and let $\psi$ be an LTL$_{-X}^{\delta}$ formula.
Suppose that $\psi$ is satisfied by $T$ under control. Then it follows from Definition~\ref{Def:control satisfy} that there exists a state $q_{2}$ of $T$ such that $tr_{\varepsilon}^{\delta}(\psi)$ is satisfied by $q_{2}$ under control.
Thus there exists a control strategy $f:Q^{+}\rightarrow 2^{A}-\{\emptyset\}$ such that
\begin{equation}\label{Eq:Lem logical connection 2}
\sigma_{2}\models tr_{\varepsilon}^{\delta}(\psi)\textrm{ for all }\sigma_{2}\in Out(q_{2},f).
\end{equation}
Moreover, it follows from Theorem~\ref{Th:bisi} that $q_{1}\sim_{\varepsilon}q_{2}$ for some state $q_{1}$ of $T_{\tau}(\Sigma)$.
Therefore, by Lemma~\ref{Lem:bisimilar state to bisimilar sequence}, there exists a control strategy $f':(Q_{\tau})^{+}\rightarrow 2^{A_{\tau}}-\{\emptyset\}$ such that for any $\sigma_{1}\in Out(q_{1},f')$,  $\sigma_{1}\sim_{\varepsilon}\sigma_{2}$ for some $\sigma_{2}\in Out(q_{2},f)$.
 Further, by Lemma~\ref{Lem:logical characterization of bisimilar sequence} and (\ref{Eq:Lem logical connection 2}), we get $\sigma_{1}\models \psi$ for any $\sigma_{1}\in Out(q_{1},f')$.
 Thus it follows from Definition~\ref{Def:control satisfy} that $\psi$ is satisfied by $q_{1}$ under control. Then $\psi$ is satisfied by $T_{\tau}(\Sigma)$ under control.   \qquad
\end{proof}

Immediately, we have the following result.

\begin{corollary}\label{Co:logical connection origin}
Given an asymptotically stable linear control system $\Sigma$ below
\begin{eqnarray*}
\Sigma : \dot{x}=\mathbf{A}x+\mathbf{B}u+\mathbf{G}v,\ \ \  x\in X,u\in U,v\in V
\end{eqnarray*}
and $\varepsilon,\delta\in\mathbb{R}_{+}$.
For any $\tau,\eta,\mu\in\mathbb{R}_{+}$ satisfying $\|e^{\mathbf{A}\tau}\|\varepsilon +\mu+\eta/2 < \varepsilon$
and for any $T\in T_{\tau, \eta,\mu}(\Sigma)$ and LTL$_{-X}$ formula $\varphi_{0}$, if $tr_{\varepsilon}^{\delta}(tr_{\delta}(\varphi_{0}))$ is satisfied by $T$ under control then $tr_{\delta}(\varphi_{0})$ is satisfied by $T_{\tau}(\Sigma)$ under control.
\end{corollary}
\begin{proof}
Follows from Definition~\ref{Def:transformation} and Theorem~\ref{Th:logical connection origin}. \qquad
\end{proof}

In this section, the functions $tr_{\delta}$ and $tr_{\varepsilon}^{\delta}$ play central roles.
We use these functions to transform LTL$_{-X}$ formula $\varphi_{0}$ to LTL$_{-X}^{\delta}$ formula $tr_{\delta}(\varphi_{0})$ and LTL$_{-X}^{\delta+\varepsilon}$ formula $tr_{\varepsilon}^{\delta}(tr_{\delta}(\varphi_{0}))$, respectively.
Similar method has been adopted in~\cite{zhang} to offer a logical characterization of $\lambda-$bisimulation~\cite{Ying}.
\section{Controller of $\Sigma$ derived from control strategy of finite abstraction}\label{Sec:construct controller}
 This section will demonstrate that, under some assumptions, given an initial state $q$ and a control strategy $f$ of finite abstraction enforcing $tr_{\varepsilon}^{\delta}(tr_{\delta}(\varphi_{0}))$, there exists a controller of $\Sigma$ derived from $q$ and $f$ which enforces $\Sigma$ satisfying $\varphi_{0}$.

\begin{definition}\label{def:controller}
Given a linear control system $\Sigma$ below
\begin{eqnarray*}
\Sigma : \dot{x}=\mathbf{A}x+\mathbf{B}u+\mathbf{G}v,\ \ \  x\in X,u\in U,v\in V
\end{eqnarray*}
and $\tau \in \mathbb{R}_{+}$.
A $\tau$-controller of $\Sigma$ is a pair $C=(X_{0},f_{c})$, where $X_{0}\subseteq X$ denotes a set of initial states and $f_{c}$ is a partial function from $X^{+}$ to $A_{\tau}$\footnote{$A_{\tau}\triangleq\{\mathbf{u}\in {\cal U}: \ the\ domain\ of\ \mathbf{u}\ is\ [0,\tau]\}$, see Definition~\ref{Def:tau}.}. The function $f_{c}$ is said to be a $\tau$-controller function.
\end{definition}

\begin{definition}\label{def:controller construct}
Given a linear control system $\Sigma$ with state space $X$
and $\varepsilon,\tau,\eta,\mu\in\mathbb{R}_{+}$.
Let $T\in T_{\tau,\eta,\mu}(\Sigma)$ and $T=(Q,A, B, \rightarrow, \mathbb{R}^{n}, H)$. Suppose that  $q_{0}\in Q$ and $f:Q^{+}\rightarrow 2^{A}-\{\emptyset\}$ is a control strategy of $T$.
Then a $\tau$-controller $C=(X_{0},f_{c})$ of $\Sigma$ is said to be derived from $q_0$ and $f$ if and only if the following hold:

(1) $X_{0}=\{q\in X: \mathbf{d}(q,q_{0})\leq\varepsilon\}$,

(2) for any $s\in X^{+}$, if there exists $s_{1}\in Out^{+}(q_{0},f)$ such that $s\sim_{\varepsilon}s_{1}$ then
$f_{c}(s)$ is defined and $\|\mathbf{x}(\tau,0,f_{c}(s),0)-a\|\leq \mu/2$ for some $s'\in Out^{+}(q_{0},f)$ and $a\in f(s')$ with $s\sim_{\varepsilon}s'$, otherwise  $f_{c}(s)$ is undefined.
\end{definition}

The following result reveals that, for any initial state $q_0$ and control strategy $f$ of finite abstraction, there exists some controller $C=(X_{0},f_{c})$ of $\Sigma$ derived from $q_0$ and~$f$.

\begin{lemma}\label{lem:controller exist}
Given a linear control system $\Sigma$ with state space $X$
and $\varepsilon,\tau,\eta,\mu\in\mathbb{R}_{+}$.
Let $T\in T_{\tau,\eta,\mu}(\Sigma)$ and $T=(Q,A, B, \rightarrow, \mathbb{R}^{n}, H)$.
Then for any $q_{0}\in Q$ and control strategy $f:Q^{+}\rightarrow 2^{A}-\{\emptyset\}$,
there exists a $\tau$-controller $C=(X_{0},f_{c})$ of $\Sigma$ derived from $q_0$ and $f$.
\end{lemma}
\begin{proof}
Let $q_0\in Q$ and let $f:Q^{+}\rightarrow 2^{A}-\{\emptyset\}$ be a control strategy of $T$.
We set
\begin{center}
$\triangle=\{s\in X^+: s\sim_{\varepsilon} s_1 \textrm{ for some } s_1\in Out^+(q_0,f)\}$.
\end{center}
So for each $s\in \triangle$, there exists $a\in A$ such that $a\in f(s_1)$ and $s\sim_{\varepsilon} s_1$ for some $s_1\in Out^+(q_0,f)$.
Moreover, for any $a\in A$, by $T\in T_{\tau,\eta,\mu}(\Sigma)$, Definition~\ref{Def:tau} and~\ref{Def:fin} and the definitions of $\cal{R}_{A_{\tau}}$ and $\mathbf{d}_h$, there exists $\mathbf{u}\in A_{\tau}$ such that $\|\mathbf{x}(\tau,0,\mathbf{u},0)-a\|\leq \mu/2$.
Thus for each $s\in\triangle$, there exists some control input $\mathbf{u}\in A_{\tau}$ such that $\|\mathbf{x}(\tau, 0,\mathbf{u},0)-a\|\leq \mu/2$ for some $s_1\in Out^+(q_0,f)$ and $a\in f(s_1)$ with $s\sim_{\varepsilon}s_1$.
Such control input may not be unique. For each $s\in\triangle$, we fix $\mathbf{u}_s\in A_{\tau}$, which is one of such control inputs.
Further, we define a partial function $f_c:X^{+}\rightarrow A_{\tau}$ as
\begin{equation*}
f_c(s)=\left\{\begin{aligned}& \mathbf{u}_s   &\textrm{  if } s\in \triangle\\
& \textrm{undefined}   &\textrm{otherwise.}
\end{aligned}\right.
\end{equation*}
It is easy to see that $(X_0, f_c)$ is a $\tau$-controller derived from $q_{0}$ and $f$, where $X_c\triangleq\{q\in X: \mathbf{d}(q,q_{0})\leq\varepsilon\}$.
\qquad
\end{proof}

To illustrate the execution of linear system $\Sigma$ with $\tau$-controller derived from $q_{0}$ and $f$, the following proposition is needed.

\begin{proposition}\label{pro:controller exist}
Given an asymptotically stable linear control system $\Sigma$ below
\begin{eqnarray*}
\Sigma : \dot{x}=\mathbf{A}x+\mathbf{B}u+\mathbf{G}v,\ \ \  x\in X,u\in U,v\in V.
\end{eqnarray*}
Let $\varepsilon,\tau,\eta,\mu\in\mathbb{R}_{+}$, $T\in T_{\tau,\eta,\mu}(\Sigma)$, $T=(Q,A, B, \rightarrow, \mathbb{R}^{n}, H)$, $q_{0}\in Q$, $f$ a control strategy of $T$ and let $C=(X_{0},f_{c})$ be a $\tau$-controller derived from $q_{0}$ and $f$.
Assume that $\|e^{\mathbf{A}\tau}\|\varepsilon +\mu+\eta/2 < \varepsilon$.
For any $s\in X^{+}$ and $q\in X$, if $f_{c}(s)$ is defined and $s[end]\xrightarrow{f_{c}(s),\mathbf{v}}_{\tau}q$ for some $\mathbf{v}\in B_{\tau}$ (see Definition~\ref{Def:tau}) then there exists $s_{1}\in Out^{+}(q_{0},f)$ such that $sq\sim_{\varepsilon}s_{1}$.
\end{proposition}
\begin{proof}
Let $s\in X^{+}$ and $q\in X$.
Suppose that $f_{c}(s)$ is defined and $s[end]\xrightarrow{f_{c}(s),\mathbf{v}}_{\tau}q$ for some $\mathbf{v}\in B_{\tau}$.
Then by Definition~\ref{def:controller construct}, there exists $s_{1}\in Out^{+}(q_{0},f)$ and $a\in f(s_1)$ such that $s\sim_{\varepsilon}s_{1}$ and
$\|\mathbf{x}(\tau,0,f_{c}(s),0)-a\|\leq \mu/2$.
So to complete the proof, it is enough to show that $s_{1}[end]\xrightarrow{a,b}q'$ and  $q\sim_{\varepsilon}q'$ for some $q'\in Q$ and $b\in B$.
By $s[end]\xrightarrow{f_{c}(s),\mathbf{v}}_{\tau}q$ and Definition~\ref{Def:tau}, we obtain
$q = \mathbf{x}(\tau,s[end],f_{c}(s),\mathbf{v}) = \mathbf{x}(\tau,s[end],0,0)+\mathbf{x}(\tau,0,f_{c}(s),0)+\mathbf{x}(\tau,0,0,\mathbf{v})$.
By Definition~\ref{Def:fin}, there exists $b\in B$ such that $\|\mathbf{x}(\tau,0,0,\mathbf{v})-b\|\leq \mu/2$.
Thus it follows that $s_{1}[end]\xrightarrow{a,b}q'$ for some state $q'\in Q$ of $T$.
Next, we show that $\mathbf{d}(q,q')\leq\varepsilon$.
By $s_{1}[end]\xrightarrow{a,b}q'$ and Definition~\ref{Def:fin}, we have $\|\mathbf{x}(\tau,s_{1}[end],0,0)+a+b-q'\|\leq\eta/2$.
It follows that
\begin{flalign*}
&\ \|\mathbf{x}(\tau,s[end],0,0)+a+b-q'\|\\
&\leq\|\mathbf{x}(\tau,s[end],0,0)+a+b-\mathbf{x}(\tau,s_{1}[end],0,0)+\mathbf{x}(\tau,s_{1}[end],0,0)-q'\|\\
&\leq\|\mathbf{x}(\tau,s[end],0,0)-\mathbf{x}(\tau,s_{1}[end],0,0)\|+\|\mathbf{x}(\tau,s_{1}[end],0,0)+a+b-q'\|\\
&\leq\|e^{\mathbf{A}\tau}\|\cdot\|s[end]-s_{1}[end]\|+\eta/2\\
&\leq\|e^{\mathbf{A}\tau}\|\cdot\varepsilon+\eta/2.
\end{flalign*}
Thus we get
\begin{align*}
\|q-q'\|&=\|q-\mathbf{x}(\tau,s[end],0,0)-a-b+\mathbf{x}(\tau,s[end],0,0)+a+b-q'\|\\
&\leq\|q-\mathbf{x}(\tau,s[end],0,0)-a-b\|+\|\mathbf{x}(\tau,s[end],0,0)+a+b-q'\|\\
&\leq\|q-\mathbf{x}(\tau,s[end],0,0)-a-b\|+\|e^{\mathbf{A}\tau}\|\cdot\varepsilon+\eta/2\\
&=\|\mathbf{x}(\tau,s[end],0,0)+\mathbf{x}(\tau,0,f_{c}(s),0)+\mathbf{x}(\tau,0,0,\mathbf{v})-\mathbf{x}(\tau,s[end]
,0,0)
\\
&\textrm{\ \ \ }-a-b\|
 +\|e^{\mathbf{A}\tau}\|\cdot\varepsilon+\eta/2\\
&\leq\|\mathbf{x}(\tau,0,f_{c}(s),0)-a\|+\|\mathbf{x}(\tau,0,0,\mathbf{v})-b\|+\|e^{\mathbf{A}\tau}\|\cdot\varepsilon+\eta/2\\
&\leq\mu+\|e^{\mathbf{A}\tau}\|\cdot\varepsilon+\eta/2
\\
&\leq\varepsilon.
\end{align*}
So by Theorem~\ref{Th:bisi} and $\|e^{\mathbf{A}\tau}\|\varepsilon +\mu+\eta/2 < \varepsilon$, we obtain $q\sim_{\varepsilon}q'$ and then $sq\sim_{\varepsilon}s_{1}q'$.
\qquad
\end{proof}

Given an initial state $q_{0}$ and a control strategy $f$ of finite abstraction $T$, the execution of system $\Sigma$ with a controller $(X_{0},f_{c})$ derived from $q_{0}$ and $f$ is described below.
We start this execution from some state $\mathbf{x}(0)\in X_{0}$ (i.e., $\mathbf{d}(\mathbf{x}(0),q_{0})\leq\varepsilon$).
Then controller function $f_{c}$ provides a control input $f_{c}(\mathbf{x}(0))$, which is applied to $\Sigma$ on the time interval $[0,\tau)$.
At time $\tau$, the system $\Sigma$ reaches at a state $\mathbf{x}(\tau)$ from $\mathbf{x}(0)$ with control input $f_{c}(\mathbf{x}(0))$ and some disturbance input.
By Proposition~\ref{pro:controller exist}, there exists a state $q_{1}$ of $T_{\tau, \eta,\mu}(\Sigma)$ such that $q_{0}q_{1}\in Out^{+}(q_{0},f)$ and $q_{0}q_{1}\sim_{\varepsilon}\mathbf{x}(0)\mathbf{x}(\tau)$.
Then controller function $f_{c}$ offers a control input $f_{c}(\mathbf{x}(0)\mathbf{x}(\tau))$, which is applied on the time interval $[\tau,2\tau)$. The process repeats in such manner.
Here we just informally describe the execution of $\Sigma$ with a controller $(X_{0},f_{c})$.
Clearly, whether such execution exists indeed depends on whether $f_c$ is defined at points in the form of $\mathbf{x}(0)\mathbf{x}(\tau)\mathbf{x}(2\tau)\cdots\mathbf{x}(n\tau)$.
This issue will be considered in Proposition~\ref{pro:exist of traj}.

The above execution produces trajectories of $\Sigma$ with controller derived from $q_{0}$ and $f$, which are formally defined below.

\begin{definition}\label{Def:trajectory of close-loop}
Given an asymptotically stable linear control system $\Sigma$ below
\begin{eqnarray*}
\Sigma : \dot{x}=\mathbf{A}x+\mathbf{B}u+\mathbf{G}v,\ \ \  x\in X,u\in U,v\in V.
\end{eqnarray*}
Let $\varepsilon,\tau,\eta,\mu\in\mathbb{R}_{+}$ and let $T\in T_{\tau,\eta,\mu}(\Sigma)$ be a finite abstraction of $\Sigma$. Suppose that $T=(Q,A, B, \rightarrow, \mathbb{R}^{n}, H)$, $q_{0}\in Q$, $f$ is a control strategy of $T$ and $C=(X_{0},f_{c})$ is a $\tau$-controller derived from $q_{0}$ and $f$.
Then $\mathbf{x}:\mathbb{R}_{+}^{0}\rightarrow X$ is said to be a trajectory of $\Sigma$ with $\tau$-controller $C$ if and only if for any $n\in\mathbb{N}$, $f_c(\sigma_{\mathbf{x}}[1,n])$ is defined and there exists $\mathbf{v}_{n}\in B_{\tau}$ (see Definition~\ref{Def:tau}) such that
$\dot{\mathbf{x}}(t)=\mathbf{A}\mathbf{x}(t)+\mathbf{B}f_{c}(\sigma_{\mathbf{x}}[1,n])(t) +\mathbf{G}\mathbf{v}_{n}(t)$
for any $t\in \mathbb{R}_{+}^{0}$ with $(n-1)\tau\leq t< n\tau$, where $\sigma_{\mathbf{x}}\triangleq\mathbf{x}(0)\mathbf{x}(\tau)\cdots$.
\end{definition}

Due to the following result, given a controller derived from $q_{0}$ and $f$, the trajectory of $\Sigma$ with this controller indeed exists.

\begin{proposition}\label{pro:exist of traj}
Given an asymptotically stable linear control system $\Sigma$ below
\begin{eqnarray*}
\Sigma : \dot{x}=\mathbf{A}x+\mathbf{B}u+\mathbf{G}v,\ \ \  x\in X,u\in U,v\in V.
\end{eqnarray*}
Let $\varepsilon,\tau,\eta,\mu\in\mathbb{R}_{+}$ such that $\|e^{\mathbf{A}\tau}\|\varepsilon +\mu+\eta/2 < \varepsilon$ and let $T\in T_{\tau,\eta,\mu}(\Sigma)$ be a finite abstraction of $\Sigma$.
Suppose that $T=(Q,A, B, \rightarrow, \mathbb{R}^{n}, H)$, $q_{0}\in Q$, $f$ is a control strategy of $T$ and $C=(X_{0},f_{c})$ is a $\tau$-controller derived from $q_{0}$ and $f$.
Then we have

(1) there exists at least one trajectory $\mathbf{x}:\mathbb{R}_{+}^{0}\rightarrow X$ of $\Sigma$  with $\tau$-controller $C$, and

(2) for any such trajectory  $\mathbf{x}:\mathbb{R}_{+}^{0}\rightarrow X$, there exists $\sigma\in Out(q_0,f)$ such that $\sigma\sim_{\varepsilon}\sigma_{\mathbf{x}}$ with $\sigma_{\mathbf{x}}=\mathbf{x}(0)\mathbf{x}(\tau)\cdots$.
\end{proposition}
\begin{proof}
(1) We demonstrate the claim below first.

\textbf{Claim.}
There exist a family of trajectories $\mathbf{x}_{n}:[0,\tau]\rightarrow X$ ($n\in \mathbb{N}$) such that for any $n\in \mathbb{N}$, $\mathbf{x}_{n-1}(\tau)=\mathbf{x}_{n}(0)$ if $n>1$, $f_c(s_n)$ is defined and for some disturbance input $\mathbf{v}_n \in B_{\tau}$, $\dot{\mathbf{x}}_n(t)=\mathbf{A}\mathbf{x}_n(t)+\mathbf{B}f_{c}(s_n)(t) +\mathbf{G}\mathbf{v}_{n}(t)$ for all $t\in [0,\tau]$, where  $s_n\triangleq \mathbf{x}_1(0)\mathbf{x}_2(0)\cdots \mathbf{x}_n(0)$.

We construct such trajectories by induction on $n$.
Let $n=1$.
Since $q_0$ is a state of finite abstraction $T$, by Definition~\ref{Def:fin}, we have $q_0\in X$.
It is clear that $q_0\sim_{\varepsilon} q_0$ and $q_0\in Out^{+}(q_0,f)$.
Thus by Definition~\ref{def:controller construct}, $f_c(q_0)$ is defined and $f_c(q_0)\in A_{\tau}$.
Further, since $\Sigma$ is forward-complete, given an arbitrary disturbance input $\mathbf{v}_1\in B_{\tau}$, there exists a trajectory $\mathbf{x}_1:[0,\tau]\rightarrow X$ such that $\mathbf{x}_1(0)=q_0$ and $\dot{\mathbf{x}_1}(t)=\mathbf{A}\mathbf{x}_1(t)+\mathbf{B}f_{c}(q_0)(t) +\mathbf{G}\mathbf{v}_1(t)$ for all $t\in [0,\tau]$.
Clearly, $\mathbf{x}_1$ is the desired one.

Suppose that $n=i+1$ and we already have trajectories $\mathbf{x}_{1},\mathbf{x}_{2},\cdots \mathbf{x}_{i}$  such that for any $k\leq i$, $\mathbf{x}_{k-1}(\tau)=\mathbf{x}_{k}(0)$ if $k>1$, $f_c(s_k)$ is defined and for some disturbance input $\mathbf{v}_k \in B_{\tau}$, $\dot{\mathbf{x}}_k(t)=\mathbf{A}\mathbf{x}_k(t)+\mathbf{B}f_{c}(s_k)(t) +\mathbf{G}\mathbf{v}_{k}(t)$ for all $t\in [0,\tau]$, where  $s_k\triangleq \mathbf{x}_1(0)\mathbf{x}_2(0)\cdots \mathbf{x}_k(0)$.
Thus by Definition~\ref{Def:tau}, we get $\mathbf{x}_i(0) \xrightarrow{f_c(s_i),\mathbf{v}_i}_{\tau} \mathbf{x}_i(\tau)$.
Further, by Proposition~\ref{pro:controller exist}, there exists $s\in Out^{+}(q_0,f)$ such that $s_i\circ\mathbf{x}_i(\tau)\sim_{\varepsilon} s$.
Thus by Definition~\ref{def:controller construct}, $f_c(s_i\circ\mathbf{x}_i(\tau))$ is defined.
Then similar to the above, there exists a trajectory $\mathbf{x}_{i+1}:[0,\tau]\rightarrow X$ such that
$\mathbf{x}_{i+1}(0)=\mathbf{x}_{i}(\tau)$ and for some $\mathbf{v}_{i+1}\in B_{\tau}$, $\dot{\mathbf{x}}_{i+1}(t)=\mathbf{A}\mathbf{x}_{i+1}(t)+\mathbf{B}f_{c}(s_{i+1})(t) +\mathbf{G}\mathbf{v}_{i+1}(t)$ for all $t\in [0,\tau]$.

Now, we return to the proof of this proposition.
By the above claim, there exist a family of  trajectories $\mathbf{x}_{i}:[0,\tau]\rightarrow X$ ($i\in \mathbb{N}$)  satisfying the conditions in the above claim.
Then based on these trajectories, a function $\mathbf{x}:\mathbb{R}_{+}^{0}\rightarrow X$ is defined as: for any $t\in\mathbb{R}_{+}^{0} $, if $t\in [(i-1)\tau, i\tau)$ for some $i\in\mathbb{N}$ then we set $\mathbf{x}(t)=\mathbf{x}_{i}(t)$.
Clearly, $\bigcup\{[(i-1)\tau, i\tau):i\in \mathbb{N}\}=\mathbb{R}_{+}^{0}$ and $[(i-1)\tau, i\tau)\cap[(j-1)\tau, j\tau)=\emptyset$ for all $i,j\in\mathbb{N}$ with $i\neq j$.
Thus for any $t\in\mathbb{R}_{+}^{0} $, there exists unique $i\in\mathbb{N}$ such that $t\in [(i-1)\tau, i\tau)$.
 So the function $\mathbf{x}$ is well-defined.
By the above claim and Definition~\ref{Def:trajectory of close-loop}, $\mathbf{x}$ is a trajectory of $\Sigma$ with $\tau$-controller $C$.

(2) Let $\mathbf{x}$ be a trajectory of $\Sigma$ with $\tau$-controller $C$ and $\sigma_{\mathbf{x}}=\mathbf{x}(0)\mathbf{x}(\tau)\cdots$.
Then by Definition~\ref{Def:trajectory of close-loop}, $f_{c}(\sigma_{\mathbf{x}}[1,n])$ is defined for any $n\in\mathbb{N}$.
Thus it follows from Definition~\ref{def:controller construct}, there exist a family of sequences $s_{n}\in Out^{n}(q_{0},f)$ ($n\in\mathbb{N}$) such that $s_{n}\sim_{\varepsilon}\sigma_{\mathbf{x}}[1,n]$ for each $n\in\mathbb{N}$.
Moreover, since $T$ is finite, the state set $Q$ is finite.
Then it is easy to check that there exists an infinite sequence $i_1 i_2\cdots \in \mathbb{N}^{\omega}$ such that for any $j\in \mathbb{N}$, $i_j<i_{j+1}$ and $s_{i_{j}}$ is a proper prefix of $s_{i_{j+1}}$.
Clearly, for any $k\in \mathbb{N}$, there exists $j\in\mathbb{N}$ such that $k<i_j$.
Furthermore, for any $j,l,k\in\mathbb{N}$, if $k<i_j$ and $k<i_l$, then $s_{i_j}[k]=s_{i_l}[k]$.
Then we define an infinite sequence $\sigma\in Q^{\omega}$ as: for any $k\in \mathbb{N}$, if $k<i_j$ for some $j\in\mathbb{N}$, then we set $\sigma[k]=s_{i_j}[k]$.
It is clear that $\sigma$ is well-defined.
Then, since $s_{i_j}\in Out^{+}(q_{0},f)$ and $s_{i_j}\sim_{\varepsilon}\sigma_{\mathbf{x}}[1,i_j]$ for all $j\in\mathbb{N}$, by Definition~\ref{Def:strategy}, we have $\sigma\in Out(q_0,f)$ and $\sigma\sim_{\varepsilon}\sigma_{\mathbf{x}}$.
\end{proof}

The following result demonstrates that under some assumptions, given an LTL$_{-X}$ formula $\varphi_{0}$ as specification, if  $\sigma\models tr_{\varepsilon}^{\delta}(tr_{\delta}(\varphi_{0}))$ for any $\sigma\in Out(q_{0},f)$, then all trajectories of $\Sigma$ with a controller derived from $q_{0}$ and $f$ satisfy specification~$\varphi_{0}$.

\begin{theorem}\label{Th:close-loop satisfy 1}
Given an asymptotically stable linear control system $\Sigma$ below
\begin{eqnarray*}
\Sigma : \dot{x}=\mathbf{A}x+\mathbf{B}u+\mathbf{G}v,\ \ \  x\in X,u\in U,v\in V.
\end{eqnarray*}
Let $\varepsilon,\tau,\eta,\mu\in\mathbb{R}_{+}$, $\varphi_{0}$ an LTL$_{-X}$ formula,  $T\in T_{\tau,\eta,\mu}(\Sigma)$ a finite abstraction of $\Sigma$, $q_{0}$ a state of $T$, $f$ a control strategy of $T$ and let $C=(X_{0},f_{c})$ be a $\tau$-controller derived from $q_{0}$ and~$f$.
Assume that $\|e^{\mathbf{A}\tau}\|\varepsilon +\mu+\eta/2 < \varepsilon$ and $\|\mathbf{x}(t)-\mathbf{x}((n-1)\tau)\|\leq \delta$ for any trajectory $\mathbf{x}$ of $\Sigma$ and for any $n\in\mathbb{N}$ and $t\in {[(n-1)\tau,} n\tau)$.
If   $\sigma\models tr_{\varepsilon}^{\delta}(tr_{\delta}(\varphi_{0}))$ for any $\sigma\in Out(q_{0},f)$,  then all trajectories of $\Sigma$ with $\tau$-controller $C$ satisfy $\varphi_{0}$.
\end{theorem}
\begin{proof}
Suppose that $\sigma\models tr_{\varepsilon}^{\delta}(tr_{\delta}(\varphi_{0}))$ for any $\sigma\in Out(q_{0},f)$.
Let $\mathbf{x}:\mathbb{R}_{+}^{0}\rightarrow X$ be a trajectory of $\Sigma$ with $\tau$-controller $C$ and $\sigma_{\mathbf{x}}=\mathbf{x}(0)\mathbf{x}(\tau)\cdots$.
Then by (2) in Proposition~\ref{pro:exist of traj}, there exists $\sigma\in Qut(q_{0},f)$ such that $\sigma\sim_{\varepsilon}\sigma_{\mathbf{x}}$.
Thus by $\sigma\models tr^{\delta}_{\varepsilon}(tr_{\delta}(\varphi_{0}))$ and Lemma~\ref{Lem:logical characterization of bisimilar sequence}, we get $\sigma_{\mathbf{x}}\models tr_{\delta}(\varphi_{0})$.
Therefore, since $\|\mathbf{x}(t)-\mathbf{x}((n-1)\tau)\|\leq \delta$ for any $n\in\mathbb{N}$ and $t\in {[(n-1)\tau,} n\tau)$, it follows from Theorem~\ref{Th:logical connection 1'} that $\mathbf{x}\models \varphi_{0}$.
\end{proof}

Now we arrive at the main result of this section.

\begin{theorem}\label{co:main}
Given an asymptotically stable linear control system $\Sigma$ below
\begin{eqnarray*}
\Sigma : \dot{x}=\mathbf{A}x+\mathbf{B}u+\mathbf{G}v,\ \ \  x\in X,u\in U,v\in V.
\end{eqnarray*}
Let $\varepsilon,\tau,\eta,\mu\in\mathbb{R}_{+}$, $\varphi_{0}$ an LTL$_{-X}$ formula and let $T\in T_{\tau,\eta,\mu}(\Sigma)$ be a finite abstraction of $\Sigma$.
Assume that $\|e^{\mathbf{A}\tau}\|\varepsilon +\mu+\eta/2 < \varepsilon$ and $\|\mathbf{x}(t)-\mathbf{x}((n-1)\tau)\|\leq \delta$ for any trajectory $\mathbf{x}$ of $\Sigma$ and for any $n\in\mathbb{N}$ and $t\in {[(n-1)\tau,} n\tau)$.
If there exists a state
$q_{0}$ and a control strategy $f$ of $T$ such that  $\sigma\models tr_{\varepsilon}^{\delta}(tr_{\delta}(\varphi_{0}))$ for any $\sigma\in Out(q_{0},f)$, then there exists some $\tau$-controller $C=(X_{0},f_{c})$ derived from $q_{0}$ and~$f$ satisfying the following conditions:

(1) there exists at least one trajectory of $\Sigma$ with $\tau$-controller $C$, and

(2) all trajectories of $\Sigma$ with $\tau$-controller $C$ satisfy $\varphi_{0}$.
\end{theorem}
\begin{proof}
Suppose that there exists a state
$q_{0}$ and a control strategy $f$ of $T$ so that  $\sigma\models tr_{\varepsilon}^{\delta}(tr_{\delta}(\varphi_{0}))$ for any $\sigma\in Out(q_{0},f)$.
Then by Lemma~\ref{lem:controller exist}, there exists a $\tau$-controller $C=(X_{0},f_{c})$ derived from $q_{0}$ and~$f$.
Further, (1) follows from Proposition~\ref{pro:exist of traj} and (2) is implied by Theorem~\ref{Th:close-loop satisfy 1}.~\end{proof}

In the above two theorems, the assumption  $\|e^{\mathbf{A}\tau}\|\varepsilon +\mu+\eta/2 < \varepsilon$ is introduced by Pola and Tabuada  to guarantee that the finite abstraction and the sample system of the given linear system are A$\varepsilon$A bisimilar (see  Theorem~\ref{Th:bisi}).
\section{Conclusion and future work}\label{Sec:discussion}
In order to provide a framework to design controller for systems affected by disturbances, Pola and Tabuada introduce finite abstractions for these systems~\cite{pola:5,pola:1}.
This paper concerns itself with the relationship between the control strategy of these abstractions and the controller of the original control systems.
Similar work has been developed for control systems without disturbances~\cite{fain:2},\cite{tab:1},\cite{tab:2}.
In these work, since finite abstractions and the original control systems share the same properties of interest,
the formal design of control systems may be equivalently performed on the corresponding finite abstractions.

This paper points out that Pola and Tabuada's finite abstraction and its original control system do not always share the identical properties described by  LTL$_{-X}$ formulae under control (see Example~\ref{example: counter1} and~\ref{example:counter2}).
Thus, if we adopt the same formula $\varphi_{0}$ as specification of control systems and finite abstractions, the formal design of the latter may not be helpful for the former.
This paper tries to fill such gap between finite abstractions and control systems with disturbances.
To this end, the specification transforming function $\lambda\varphi.tr^{\delta}_{\varepsilon}(tr_{\delta}(\varphi))$ is introduced, which transforms a specification for control systems to one for finite abstractions.
We illustrate that under some assumption, given an initial state $q$ and a control strategy $f$ of finite abstraction enforcing $tr^{\delta}_{\varepsilon}(tr_{\delta}(\varphi_{0}))$, then there exists a controller derived from $q$ and $f$ such that the trajectories of $\Sigma$ with this controller satisfy $\varphi_0$ (see Theorem~\ref{co:main}).
In another paper~\cite{zhang:2}, we also provide an algorithm to obtain an initial state and a control strategy which enforces a given finite abstraction satisfying desired specification.
These results indicate that Pola and Tabuada's abstractions may be a useful tool in the formal design of control systems with disturbance inputs.

However, this paper just proves the existence of controller derived from the given initial state $q$ and control strategy $f$, but does not offer the construction of such controller.
In other words, Definition~\ref{def:controller construct} just tells us what is a controller derived from $q$ and $f$, but does not provide a way to obtain it.
Clearly, it is a topic worthy of further study that how to obtain such controller.

%\section*{Acknowledgment}


\ifCLASSOPTIONcaptionsoff
  \newpage
\fi

\end{document}